\documentclass[a4paper,11pt]{amsart}

\usepackage{amssymb,mathrsfs,amsthm,bm}
\usepackage{mathtools}
\mathtoolsset{showonlyrefs}
\usepackage[utf8]{inputenc}
\usepackage[margin=1in]{geometry}
\usepackage{hyperref,hypcap}
\hypersetup{colorlinks=true,allcolors=blue,bookmarksdepth=3}
\usepackage[inline]{enumitem}
\usepackage{microtype}
\usepackage{soul}

\usepackage[backend=bibtex,style=alphabetic,sorting=nyt,firstinits=true,doi=false,isbn=false,url=false,maxbibnames=99,maxcitenames=99]{biblatex}
\AtEveryBibitem{\clearlist{language}}
\usepackage{longfbox}

\theoremstyle{plain}
\newtheorem{theorem}{Theorem}[section]
\newtheorem{lemma}[theorem]{Lemma}
\newtheorem{proposition}[theorem]{Proposition}

\theoremstyle{definition}
\newtheorem{definition}[theorem]{Definition}
\theoremstyle{remark}
\newtheorem{remark}[theorem]{Remark}

\numberwithin{equation}{section}
\numberwithin{figure}{section}
\numberwithin{table}{section}
\newcommand{\dd}{\mathop{}\!\mathrm{d}}
\newcommand{\dist}{\mathrm{dist}}
\newcommand{\NN}{\mathbb{N}}
\newcommand{\R}{\mathbb{R}}
\newcommand{\Z}{\mathbb{Z}}
\newcommand{\T}{\mathbb{T}}

\DeclareMathOperator{\diam}{diam}
\DeclareMathOperator{\supp}{supp}

\author{Nicolas Burq}
\address[Nicolas Burq]{Laboratoire de Mathématiques d'Orsay, UMR CNRS 8628, Orsay, France, and Institut universitaire de France}
\email{nicolas.burq@universite-paris-saclay.fr}

\author{Pierre Germain} 
\address[Pierre Germain]{Department of Mathematics, Huxley building, South Kensington campus, Imperial College London, London SW7 2AZ, United Kingdom}
\email{pgermain@ic.ac.uk}

\author{Massimo Sorella} 
\address[Massimo Sorella]{Department of Mathematics, Huxley building, South Kensington campus, Imperial College London, London SW7 2AZ, United Kingdom}
\email{msorella@ic.ac.uk}

\author{Hui Zhu}
\address[Hui Zhu]{New York University Abu Dhabi, Division of Science}
\email{hui.zhu@nyu.edu}

\begin{document}

\title{Trace and Observability Inequalities for Laplace Eigenfunctions on the Torus}

\begin{abstract}
We investigate trace and observability inequalities for Laplace eigenfunctions on the $d$-dimensional torus $\T^d$, with respect to arbitrary Borel measures $\mu$. Specifically, we characterize the measures $\mu$ for which the inequalities  
\begin{equation*}
    \int |u|^2 \dd\mu \lesssim \int |u|^2 \dd x \quad \text{(trace)}, \qquad 
\int |u|^2 \dd\mu \gtrsim \int |u|^2 \dd x \quad \text{(observability)}
\end{equation*}
hold uniformly for all eigenfunctions $u$ of the Laplacian. Sufficient conditions are derived based on the integrability and regularity of $\mu$, while necessary conditions are formulated in terms of the dimension of the support of the measure.

These results generalize classical theorems of Zygmund and Bourgain--Rudnick to higher dimensions.
Applications include results in the spirit of Cantor--Lebesgue theorems, constraints on quantum limits, and control theory for the Schrödinger equation.

Our approach combines several tools: the cluster structure of lattice points on spheres; decoupling estimates; and the construction of eigenfunctions exhibiting strong concentration or vanishing behavior, tailored respectively to the trace and observability inequalities.
\end{abstract}

\maketitle


\section{Introduction}

\subsection{Trace and observability inequalities}

For $d\ge 2$ and $\lambda \in \sqrt{\NN}= \{\sqrt{n} : n \in \NN\}$, let $\mathcal{S}^{d-1}_\lambda = \Z^d \cap \lambda \mathbb{S}^{d-1}$.
The eigenspace associated to the eigenvalue $-4 \pi^2 \lambda^2$ of the Laplacian $\Delta$ on the torus $\T^d = \R^d / \Z^d$ is given by
\begin{equation}
    E_\lambda = \biggl\{ \sum_{k \in \mathcal{S}^{d-1}_\lambda} \widehat{u}_k e^{2\pi i k \cdot x} :  \widehat{u}_k \in \mathbb{C}, \ \forall  k \in \mathcal{S}^{d-1}_\lambda \biggr\}.
\end{equation}

Let $\mathcal{P}(\T^d)$ be the set of Borel probability measures on $\T^d$.
For $\mu \in \mathcal{P}(\T^d)$, let
$\mathfrak{C}_\mu$ resp.\ $\mathfrak{c}_\mu$ be the infimum of all $\mathfrak{C} \in [0, \infty]$ resp.\ $\mathfrak{c} \in [0,\infty]$ such that for all $u \in \bigcup_{\lambda \in \sqrt{\NN}} E_\lambda$, there holds
\begin{equation}
\label{eq::trace-obs-ineq-def}
    \int_{\T^d} |u|^2 \dd \mu \leq \mathfrak{C} \int_{\T^d} |u|^2 \dd x,
    \quad
    \text{resp.}
    \quad
    \int_{\T^d} |u|^2 \dd x \leq \mathfrak{c} \int_{\T^d} |u|^2 \dd \mu.
\end{equation}
If the first inequality holds for some $\mathfrak{C} < \infty$, then we say that the \textit{trace inequality} (of toral eigenfunctions) holds true for $\mu$. 
If the second inequality holds for some $\mathfrak{c} < \infty$, then we say that the \textit{observability inequality} (of toral eigenfunctions) holds true for $\mu$.

\subsection{State of the art}
Known results are as follows: the  trace inequality holds if
\begin{itemize}
    \item $d=2$ and $\dd \mu = f \dd x$, where $f \in L^2(\T^d)$.
    This was proved by Zygmund \cite{Zygmund1974}.
    \smallskip
    \item $d \geq 2$ and $\widehat{\mu} \in \ell^{\frac d {d-1}}$, where $\widehat{\mu} : \Z^d \to \mathbb{C}$ denotes the Fourier transform of $\mu$.
    This is a consequence of the work by Jakobson \cite{Jakobson1997quantum}, see also  Aïssiou \cite{Aissiou2013semiclassical}.
    \smallskip
    \item $d=2,3$ and $\mu$ is the superficial measure on a nonempty open subset of a smooth ($d=2$) or analytic $(d=3$) hypersurface with non-vanishing Gaussian curvature.
    This was proved by Bourgain and Rudnick \cite{BourgainRudnick2009restriction}.
    \smallskip
    \item $d \geq 2$ and $\mu$ is the superficial measure on a nonempty open subset of a rational   hyperplane (i.e.\  a hyperplane whose normal vector has rationally dependent coordinates).
    This is due to Huang and Zhang \cite{HuangZhang}.
\end{itemize}

The  observability inequality holds  if
\begin{itemize}
    \item $d=2$ and $\dd \mu = \mathbf{1}_E \dd x$ where $E$ is a Borel set of positive Lebesgue measure.
    This is due to Zygmund \cite{Zygmund1972}, and implies the same for $\dd \mu = f \dd x$ where $f \in L^1(\T^d)$.
    \smallskip
    \item $d \geq 2$ and $\dd \mu = \mathbf{1}_E \dd x$ where $E$ is a nonempty open set.
    This is due to Connes \cite{Connes1976coefficients} and implies the same for continuous densities.
\end{itemize}

The observability inequality does not hold if 
{$\supp \mu$} is contained in the zero set of an eigenfunction. This obstruction prompts the definition of the \textit{semiclassical observability inequality}, which requires the existence of $\mathfrak{c} \in [0,\infty)$ and $\lambda_0 \in [0,\infty)$ such that the second estimate in \eqref{eq::trace-obs-ineq-def} holds for all $u \in \bigcup_{\lambda_0 < \lambda \in \sqrt{\NN}} E_\lambda$.
The optimal constant $\mathfrak{c}$ as $\lambda_0 \to \infty$ is denoted $\mathfrak{c}_\mu^{\operatorname{sc}}$.
\begin{itemize}
    \item The semiclassical observability inequality holds if $d=2,3$ and $\mu$ is the superficial measure on a nonempty open subset of an analytic hypersurface with non-vanishing Gaussian curvature, a result by Bourgain and Rudnick \cite{BourgainRudnick2009restriction}.
\end{itemize}

By contrast, the  semiclassical observability inequality does not hold if $ \mu$ is supported on a rational hyperplane, or on a cylinder of the type $\{ (x_1,\dots x_d) \in \T^d : (x_1,\dots x_k) \in Z \}$, where $k \leq d-1$ and $Z$ the zero set of an eigenfunction in $\T^k$. Indeed, there are eigenfunctions with arbitrarily large eigenvalues vanishing on such sets. The  semiclassical observability inequality does not hold either for measures supported on general hyperplanes, as a consequence of a Diophantine approximation argument.

Finally, the very recent manuscript \cite{BurqZhu} by two of the authors of the present paper addresses the question of observability for the time-dependent Schr\"odinger equation (with potential) on the torus, which is closely related to the observability of eigenfunctions. This paper shows that, under rather weak conditions, trace inequality implies observability inequality. 
In particular, it shows that (weak forms of) uniform $L^p$ bounds ($p>2$) for eigenfunctions imply observability from general Borel sets with positive measures.

\subsection{Broader context}

The trace inequality studied in the present manuscript is uniform in $\lambda$.
Some toral eigenfunction bounds, however, do grow with $\lambda$, and a complementary line of research aims at finding optimal estimates in this regime  \cite{Bourgain1993,BourgainDemeter,HuangZhang}. One can also consider quasimodes, i.e.\  functions with narrow spectral localization \cite{GermainMyerson,Germain}.

This trace inequality has a natural counterpart in Euclidean spaces.
Choosing $\mu$ as a probability measure on $\R^d$, the problem is to find conditions on $\mu$ such that the Fourier extension operator $g \mapsto \mathcal{F}(g \dd \sigma_{\mathbb{S}^{d-1}})$ (where $\mathcal{F}$ is the Fourier transform on $\R^d$) is a bounded map from $L^2(\mathbb{S}^{d-1})$ to $L^2(\dd \mu)$ (note that the parameter $\lambda$ disappears by scale invariance).
A sufficient condition on $\mu$ was conjectured by Mizohata and Takeuchi and was very recently disproved \cite{Cairo}, see also \cite{CarberyIliopoulouWang} and references therein for positive results.

Eigenfunction bounds have also been investigated for general compact Riemannian manifolds.
Universal estimates are known \cite{BurqGerardTzvetkov, Sogge} and are usually saturated on spheres. 

Similarly, one can ask about observability inequalities for general compact Riemannian manifolds. The answer depends strongly on the global geometry of the manifold.
For instance, observability inequalities from nonempty open sets hold true on surfaces with negative curvature \cite{Jin2017,DyatlovJinNonnenmacher}, but not on spheres, where  \emph{highest weight spherical harmonics} may concentrate on great circles.

Finally, a last line of research consists in adding a potential on the torus and asking for observability inequalities; for this we refer to \cite{BurqZhu}.

\subsection{Necessary conditions}

{We now present our main results. The reader may refer to Subsection~\ref{sec::notations} for notations and conventions used in their statements and throughout the paper.}

Recall that a Borel measure $\mu$ on $\T^d$ is upper $\alpha$-regular if 
\begin{equation*}
    \mu(B_r(x_0)) \lesssim r^\alpha, \quad \forall x_0 \in \T^d,\ \forall r \in (0,1),
\end{equation*}
where $B_r(x_0)$ is the (geodesic) ball on $\T^d$ centered at $x_0$ with radius $r$.

We obtain the following necessary conditions for the trace and observability inequalities.

\begin{theorem}[Necessary condition for the trace inequality]
\label{thm:trace-necessary}
   Let $d \geq 2$. If the trace inequality holds for $\mu \in \mathcal{P}(\T^d)$, then $\mu$ is upper $(d-2)$-regular, and hence the Hausdorff dimension of its support is $\geq d-2$.
\end{theorem}

This result is proved in Section~\ref{sec:trace-nec}, see that section for more precise and improved necessary conditions on $\mu$ in low dimensions $d=2,3,4$.

\begin{theorem}[Necessary condition for the observability inequality]
\label{thm:obs-necessary}
Let $d \geq 2$. If the observability inequality holds for $\mu \in \mathcal{P}(\T^d)$, then the Minkowski dimension of its support is $\geq d-2$.
\end{theorem}

This result is proved in Section~\ref{sec:observability-nec}, see that section for more necessary conditions on measure $\mu$ supported on $d-2$ and $d-3$ dimensional hypersurfaces. 
For both results,  the number $d-2$ corresponding to the power law in $\# \mathcal{S}^{d-1}_\lambda \sim_d \lambda^{d-2}$ plays a crucial role (this equivalence is true in dimension $\geq 5$, but fluctuations appear in smaller dimensions).

\subsection{Sufficient conditions} 

Is the $(d-2+\epsilon)$-regularity (where $\epsilon > 0$) also sufficient for the trace inequality to hold?
This is conceivable since we do not have counter-examples. 
Notice that the trace inequality for $(d-2+\epsilon)$-regular measures implies  two major conjectures in the field (see Subsection~\ref{openquestions} for more on various open questions):
\begin{itemize}
    \item The trace inequality for hypersurfaces \cite{BourgainRudnick2012};
    \smallskip
    \item The estimate for eigenfunctions: $\| u \|_{L^p} \lesssim \| u \|_{L^2}$ if $p < \frac{2d}{d-2}$ proposed in \cite{Bourgain1993,Bourgain2001}.
\end{itemize}
As far as the observability inequality goes, some curvature assumptions are certainly needed beyond the Minkowski dimension condition, as can be seen in the case of flat hypersurfaces.

We now turn to the sufficient conditions we are able to prove.

\subsubsection{Fourier decay}

In Sections \ref{sec:dimension2} and \ref{sectionpointwiseFourier}, using point counting techniques and dimension induction, we establish the trace and observability inequalities when certain Fourier decay conditions hold for $\mu$.
In the following theorems $(\widehat{\mu}_k)_{k\in\Z^d}$ denote Fourier series of $\mu$.

\begin{theorem}
\label{thm::fourier-decay}
    Let $d \geq 2$. The trace and observability inequalities hold if $\mu \in \mathcal{P}(\T^d)$ satisfies, for some $\epsilon>0$:
    \begin{equation}
    \label{eq::mu-condition-fourier-decay}
        |\widehat{\mu}_k| \lesssim |k|^{2-d-\epsilon}.
    \end{equation}
\end{theorem}

Note that the decay rate $d-2$ in the previous theorem is sharp, see Remark \ref{lastremark}. If $d \ge 5$, this theorem can be improved to reach a condition which shares the same scaling invariance as $(d-2)$-regular measures.

\begin{theorem}
\label{thm::fourier-dyadic}
    For $d \ge 5$, the trace and observability inequalities hold if $\mu \in \mathcal{P}(\T^d)$ satisfies
    \begin{equation}
    \label{eq::mu-condition-dyadic}
        \sum_{j \ge 0} 2^{j(d-2)} \sup_{|k| \in [2^j,2^{j+1}]} |\widehat{\mu}_k| < \infty.
    \end{equation}
\end{theorem}

\subsubsection{Sobolev regularity}

In Section~\ref{sectionsobolev}, we turn to Sobolev functions, for which a key tool is the $\ell^2$ decoupling theorem of Bourgain and Demeter \cite{BourgainDemeter}.

\begin{theorem}
    Let $d \geq 2$. The trace and observability inequalities hold for $\mu \in \mathcal{P}(\T^d)$ if for some $\epsilon > 0$, there holds
    \begin{equation}
        \dd \mu = f \dd x,
        \quad
        f \in W^{\epsilon,\frac{d+1}{2}}(\T^d).
    \end{equation}
\end{theorem}

When $d=2$, the condition may be relaxed to $f \in W^{\epsilon,p}(\T^d)$ for any $p \ge 1$. In applications, it is often most relevant to choose $f = \mathbf{1}_E$ in the above theorem, and the question becomes to characterize sets $E$ such that $\mathbf{1}_E \in W^{\epsilon,p}(\T^d)$. It is explored in Appendix \ref{sec:appendixA}.

\subsection{Notations}
\label{sec::notations}

The following notations are used throughout the paper.
\begin{itemize}
    \item For any set of parameters $A$, we write $f \lesssim_A g$ resp.\ $f \gtrsim_A g$ if, for some finite constant $C>0$ depending solely on $A$, the inequality $f \le C g$ resp.\ $f \ge C g$ holds.
    We write $f \sim_A g$ if both $f \lesssim_A g$ and $f \gtrsim_A g$ hold.
    We will also use $C_A$, without further specification, to denote a finite and strictly positive constant depending solely on $A$.
    Therefore, the relations $f \lesssim_A g$ and $f \gtrsim_A g$ are respectively equivalent to $f \le C_A g$ and $f \ge C_A g$.
    \item We will use $\|\cdot\|_{L^2(\dd \mu)}$ to denote the $L^2$-norm with respect to the measure $\mu$.
    We will use $\|\cdot\|_{L^2}$ to denote the $L^2$-norm with respect to the Lebesgue measure on the torus $\T^d$.
    \item For any distribution $u$ on $\T^d$, we denote by $\widehat{u}_k$ ($k \in \Z^d$) its Fourier coefficients:
    \begin{equation*}
        \widehat{u}_k = \int_{\T^d} e^{-2\pi i k\cdot x} u(x) \dd x.
    \end{equation*}
\end{itemize}

\section*{Acknowledgements}

The research of Nicolas Burq has received funding from the European Research Council (ERC) under the
European Union's Horizon 2020
research and innovation programme (Grant agreement 101097172 - GEOEDP). 
Pierre Germain was supported by the
Simons Foundation Collaboration on Wave Turbulence, a start up grant from Imperial College and a Wolfson fellowship. 
Massimo Sorella acknowledges support from the Chapman Fellowship at Imperial College London and would like to thank Luigi De Rosa for insightful discussions.
Hui Zhu was partially supported by the
Simons Foundation Collaboration on Wave Turbulence.

\section{Perspectives}

\label{sectionperspectives}

\subsection{Implications of the  trace and observability inequalities}

The trace and observability inequalities studied here are not only intrinsically interesting but also have various implications.

\subsubsection{Cantor--Lebesgue theorems and spherical summation of Fourier series} 

If $E$ is a measurable subset of $\T^d$,  observability for the measure $\mathbf{1}_E \dd x$ is defined as the following property
\begin{equation} \label{O} \tag{O}
    \sum_{|k|^2=n} |c_k|^2 \lesssim \biggl\| \sum_{|k|^2=n} c_k e^{2\pi i k \cdot x} \biggr\|_{L^2(E)}^2, \quad \forall n \in \NN
\end{equation}

We saw that this statement was known if $d=2$ and $E$ is Borel with positive measure \cite{Zygmund1972} or $d \geq 3$ and $E$ is open (and nonempty) \cite{Connes1976coefficients}. 
It is also true if $\mathbf{1}_E$ is in $W^{\epsilon,\frac{d+1}{2}}(\T^d)$ as proved in the present article. For completeness, we give a non-trivial example of such a set in Appendix \ref{sec:appendixA}, see also \cite{Lomb19} for the proof of the von Koch snowflake $S \subset \R^2$ such that $\mathbf{1}_S \in W^{s,1 } (\R^2)$ for any $s \in \bigl(0, 2 - \frac{\ln 4}{\ln 2} \bigr)$.
The observability \eqref{O} has the following corollary
\begin{equation*} 
    \biggl\| \sum_{|k|^2=n} c_k e^{2\pi i k \cdot x} \biggr\|_{L^2(E)} \overset{n \to \infty}{\longrightarrow} 0
    \quad\implies\quad
    \sum_{|k|^2=n} |c_k|^2 \overset{n \to \infty}{\longrightarrow} 0
\end{equation*}

Cantor--Lebesgue theorems refer to a pointwise version of the above, namely
\begin{equation} \label{CL} \tag{CL}
    \sum_{|k|^2=n} c_k e^{2\pi i k \cdot x} \overset{n \to \infty}{\longrightarrow} 0,\ \forall x \in E
    \quad
    \implies
    \quad
    \sum_{|k|^2=n} |c_k|^2 \overset{n \to \infty}{\longrightarrow} 0.
\end{equation}
As observed in \cite{Zygmund1972}, Egorov's theorem implies that condition \eqref{O}, when satisfied on measurable sets of positive Lebesgue measure, yields \eqref{CL} for such sets. In particular, when $d=2$, condition \eqref{CL} holds for every set with positive Lebesgue measure. Separately, Connes \cite{Connes1976coefficients} applied Baire's category theorem to show that \eqref{O} on open sets implies \eqref{CL} for open sets. Thus, if 
$d\ge 3$, then \eqref{CL} holds for all nonempty open sets.

Cantor--Lebesgue theorems are a crucial step in proving uniqueness of Fourier series by spherical summation. 
A key result in this theory is the following
\begin{equation*}
    \lim_{n \to \infty} \sum_{|k|^2 \leq n} c_k e^{2\pi i k \cdot x} \overset{n \to \infty}{\longrightarrow} f(x),\ \forall x \in \T^d
    \quad
    \implies
    \quad
    c_k = \widehat{f}_k,\ \forall k \in \Z^d.
\end{equation*}
under the assumption that $f \in L^1(\T^d)$ and everywhere finite.
This is due to Ash and Wang \cite{AshWang2000} following Bourgain \cite{Bourgain1996}. How much the assumption \emph{``for all $x \in \T^d$''} can be relaxed to \emph{``for all $x \in E$''}, for some set $E$, is not well-understood; the complements of such sets $E$ are known as \emph{sets of uniqueness} for spherical summation.

We refer to the reviews \cite{Cooke2,AshWang2007} for more on Cantor--Lebesgue theorems and uniqueness of Fourier series by spherical summation.

\subsubsection{Trace and observability for the Schr\"odinger equation on the torus}

By using orthogonality in time of $e^{i 4\pi^2 n t}$ with $n \in \NN$, it follows that (see e.g., \cite{GermainMoyanoZhu2024vanishing,BurqZhu}):
\begin{equation}
    \mathfrak{c}_\mu^{-1}\| u_0\|_{L^2}^2 
    \leq \int_{0}^{1} \int_{\T^d} |e^{it\Delta/(2\pi)} u_0 |^2 \dd \mu \dd t 
    \leq \mathfrak{C}_\mu \| u_0\|_{L^2}^2.
\end{equation}
If $\dd \mu = \mathbf{1}_E \dd x$, then the Hilbert Uniqueness Method of Lions \cite{Lions} gives the equivalence of the exact controllability of the Schr\"odinger equation from $E$ with the non-vanishing of the constant on the left-hand side. We refer to \cite{BurqZhu} for an overview of the question of the observability of the Schr\"odinger equation on the torus.

\subsubsection{Quantum measures} 

By definition, a quantum measure on $\T^d$ is a  weak-$*$ limit, in the space of Radon measures, of a sequence $|u_n|^2 \dd x$, where $u_n \in E_{\lambda_n}$ are $L^2$-normalized and $\lambda_n \to \infty$ as $n \to \infty$.
On $\T^d$, quantum measures are absolutely continuous with respect to the Lebesgue measure due to Bourgain, see \cite{Jakobson1997quantum} where additional properties are also proved. It is natural to ask whether trace and observability estimates for eigenfunctions have a counterpart for quantum measures. Formally, for a quantum measure $f$, it is tempting to write 
\begin{equation} 
\label{eq:quantum-inequality}
    \mathfrak{c}_\mu^{-1} \le \liminf_{\sqrt{\NN} \ni \lambda \to \infty} \inf_{u \in E_\lambda} \frac{\int_{\T^d} |u|^2 \dd \mu}{\int_{\T^d} |u|^2 \dd x}
    \leq \int_{\T^d} f \dd \mu  
    \leq \limsup_{\sqrt{\NN} \ni \lambda \to \infty} \sup_{u \in E_\lambda} \frac{\int_{\T^d} |u|^2 \dd \mu}{\int_{\T^d} |u|^2 \dd x} \le \mathfrak{C}_\mu,
\end{equation}
but this remains to be justified carefully --- even giving a meaning to $\int_{\T^d} f \dd \mu$ is not obvious if $\dd \mu$ is not given by an $L^\infty$ density. This can be done in the framework which we now describe.

Suppose that $X$ is a Banach space of distributions on $\T^d$. We say that the space $X$ satisfies the {\em uniform trace inequality} if for all $u \in \bigcup_{\lambda \in \sqrt{\NN}} E_\lambda$ and $\varphi \in X$, there holds
\begin{equation}
    |\langle |u|^2, \varphi \rangle| \lesssim_X  \| \varphi \|_X.
\end{equation}
By this and using $C^\infty(\T^d) \subset X'$, a sequence of $L^2$-normalized eigenfunctions $(u_n)_{n}$ satisfies $\| |u_n|^2 \|_{X'} \lesssim 1$.
By the Banach--Alaoglu theorem, the sequence $( |u_n|^2 )_{n}$ converges, up to a subsequence, with respect to the weak-$*$ topologies in both $X'$ and the space of Radon measures on $\T^d$ to some $f \in L^1$.
This is the density of a quantum measure.

This proves the validity of \eqref{eq:quantum-inequality} when $\mu \in X$, for in this case one has
\begin{equation}
    \langle f, \mu \rangle = \lim_{n \to \infty} \int_{\T^d} |u_n|^2 \dd \mu.
\end{equation}

An explicit example of $X$ is given by $X = W^{\epsilon,\frac{d+1}{2}}(\T^d)$ (Theorem \ref{thm:sobolev}).
A natural function space for which the uniform trace inequality holds is the completion of $C^\infty(\T^d)$ with respect to the norm $\lvert\lvert\lvert \varphi \rvert\rvert\rvert = \mathfrak{C}_\varphi$. See \cite{BurqZhu} for further developments of this idea.

\subsection{Open questions}
\label{openquestions}

We gather here some open questions if $d \geq 3$ in increasing order of difficulty and describe how they relate to each other.
\begin{enumerate}
    \item Does the observability inequality hold on $\T^d$ ($d \ge 3$) if $\dd \mu = \mathbf{1}_E \dd x$, where $E$ is a set of positive Lebesgue measure?
    \item Does the trace inequality hold on $\T^d$ for the superficial measure of
    \begin{enumerate}
        \item a smooth hypersurface if $d \geq 4$ (conjectured in \cite{BourgainRudnick2012}); or
        \item a smooth surface of codimension two if $d \geq 5$?
    \end{enumerate}
    \item Same question for the observability inequality under an appropriate curvature condition.
    \item Does there exist a \emph{strict Young function} $F$ (which is by definition a positive, finite, convex and superlinear function) such that for all $L^2$-normalized $u \in \bigcup_{\lambda \in \sqrt{\NN}} E_\lambda$,
    \begin{equation*}
    \label{eq::uniform-integrability}
        \int_{\T^d} F(|u|^2) \dd x \lesssim 1?
    \end{equation*}
    The validity of this statement (for some $F$) is conjectured in \cite{BurqZhu} and implies the existence of another strict Young function $G$ such that the trace inequality holds for all measures of the type $\dd \mu = f \dd x$ provided that
    \begin{equation*}
        \int_{\T^d} G(f) \dd x < \infty.
    \end{equation*}
    Indeed, letting $G$ be the Legendre transforms of $F$, then Young's inequality gives:
    \begin{equation*}
        \int_{\T^d} |u|^2 f \dd x 
        \leq \int_{\T^d} F(|u|^2) \dd x 
        + \int_{\T^d} G(f) \dd x.
    \end{equation*}
    \item Does the trace inequality hold on $\T^d$ for measures of the type $f \dd x$ with $f \in L^p(\T^d)$ with $p > \frac d2$? This was conjectured in \cite{Bourgain1993,Bourgain2001}.
    \item Does the trace inequality hold on $\T^d$ for general $(d-2+\epsilon)$-regular measures?  This question arises naturally in light of the results obtained in the present paper. 
\end{enumerate}

The implications between positive answers to the above questions are as follows
\begin{equation*}
    (6) \implies (5) \implies (4) \implies (1) \qquad \mbox{and} \qquad (6) \implies (2).
\end{equation*}
The only nontrivial implication is $(4) \implies (1)$ which is proved in \cite{BurqZhu}.
Question (6) is the most difficult. 
It would already be very interesting to have more non-trivial examples of $(d-2+\epsilon)$-regular measures for which the trace inequality holds. The only example we currently know of is provided by Theorem \ref{thm:decay-fourier}, see Remark \ref{littleremark}

\section{Preliminaries}
\label{sec:prel}

\subsection{Classical definitions}

Let $\mathcal{M}(\T^d)$ be the space of signed Radon measures on $\T^d$.
By the Riesz representation theorem, we have the duality $\mathcal{M}(\T^d) \simeq C(\T^d)'$ where $C(\T^d)$ is the space of continuous \emph{real-valued} functions on $\T^d$.
The set of Borel probability measures on $\T^d$ is
\begin{equation*}
    \mathcal{P}(\T^d) \coloneqq \bigl\{ \mu \in \mathcal{M}(\T^d) : \mu \geq 0,\ \mu(\T^d) = 1 \bigr\}.
\end{equation*}

For $E \subset \T^d$ and $s \geq 0$, the $s$-dimensional Hausdorff measure of $E$ is defined by
\begin{equation*}
    \mathcal{H}^s(E) \coloneqq \sup_{\delta >0} \inf \biggl\{ \sum_\alpha (\operatorname{diam} U_\alpha)^s : E \subset \bigcup_\alpha U_\alpha,\ \operatorname{diam} U_\alpha < \delta \biggr\}.
\end{equation*}
The \emph{Hausdorff dimension} of $E$ is then defined by
\begin{equation*}
    \dim_{\mathrm{H}}(E) \coloneqq \inf \{ s \geq 0 : \mathcal{H}^s(E) = 0 \}.
\end{equation*}

Let $N_\delta(E)$ denote the minimal number of closed balls of radius $\delta > 0$ needed to cover $E$. Then the \emph{upper Minkowski dimensions} is defined as
\begin{equation*}
    \overline{\dim}_{\mathrm{M}}(E) \coloneqq \limsup_{\delta \to 0} \frac{\log N_\delta(E)}{-\log \delta}.
\end{equation*}

\subsection{Elementary properties of the trace and observability inequalities}

The following lemma gives the basic structures of the set of measures satisfying the trace and observability inequalities respectively.
Its proof is straightforward and will be omitted.
For convenience in the statements, we extend the trace and observability inequalities to general Borel measures in an obvious way (i.e.\ we abandon the requirement that the measure be a probability measure).

\begin{lemma} 
    Let $\mu,\nu$ be positive Borel measures on $\T^d$. Then
    \begin{itemize}
        \item $\mathfrak{C}_\mu \geq |\mu|$ and $\mathfrak{c}_\mu \geq |\mu|^{-1}$. 
        \smallskip
        \item $\mathfrak{C}_\mu \leq \mathfrak{C}_\nu$ and $\mathfrak{c}_\mu \geq \mathfrak{c}_\nu$ if $\mu \leq \nu$.
        \smallskip
        \item $\mathfrak{C}_{a \mu + b \nu} \leq a \mathfrak{C}_{\mu} + b \mathfrak{C}_\nu$ and $\mathfrak{c}_{a \mu + b \nu}^{-1} \leq  a \mathfrak{c}_{ \mu}^{-1} + b \mathfrak{c}_\nu^{-1} $ if $a,b \geq 0$. 
        \smallskip
        \item $\mathfrak{C}_{T_{x_0} \# \mu} = \mathfrak{C}_\mu$ and $\mathfrak{c}_{T_{x_0} \# \mu} = \mathfrak{c}_\mu$ where $T_{x_0}$ is the translation by $x_0 \in \R^d$.
        \smallskip
        \item $\mathfrak{C}_{\mu * \nu} \leq \min(|\nu| \mathfrak{C}_{\mu}, |\mu| \mathfrak{C}_{\nu})$ and $\mathfrak{c}_{\mu * \nu}^{-1} \ge \min(|\nu|\mathfrak{c}_{\mu}^{-1},|\mu|\mathfrak{c}_{\nu}^{-1})$.
\end{itemize}
\end{lemma}

\subsection{Some results on lattice points on spheres}

In the remaining of the paper, we denote
\begin{equation}
    N_d(\lambda) = \# \mathcal{S}^{d-1}_\lambda,
    \qquad
    N_d(\lambda,r) = \sup_{x_0 \in \R^d} \# \bigl( \mathcal{S}^{d-1}_\lambda \cap B_r(x_0) \bigr).
\end{equation}
In other words $N_d(\lambda)$ is the number of integer points on $\lambda \mathbb{S}^{d-1}$ and $N_d(\lambda,r)$ is the largest possible number of integer points on spherical caps of radius $\le r$ on $\lambda \mathbb{S}^{d-1}$.

The following lemma will be useful.

\begin{lemma}[Sphere with more points than average]
\label{lem::lattice-number-lower-bound}
    For any $R>0$, there exists $\lambda \in [R-1,R+1] \cap \sqrt{\NN}$ such that $N_d(\lambda) \gtrsim_d \lambda^{d-2}$.
\end{lemma}
\begin{proof}
    This is an immediate consequence of the pigeonhole principle.
\end{proof}

Next, we turn to classical results on the number of lattice points on spheres centered at $0$, for which we refer to \cite{Grosswald, IwaniecKowalski,Bateman} and references therein.

\begin{theorem}[Cardinality of $\mathcal{S}^{d-1}_\lambda$]
\label{thm::cardinality-lattice-sphere}
    The following estimates hold:
    \begin{itemize}
        \item $N_2(\lambda) \lesssim \exp\bigl\{C \ln \lambda/\ln \ln \lambda\bigr\}$ for some $C>0$.
        \smallskip
        \item $N_3(\lambda) \lesssim \lambda \ln\lambda \ln\ln\lambda$ and $\limsup_{\lambda \to \infty} N_3(\lambda) / (\lambda \ln \ln \lambda) > 0$.
        \smallskip
        \item $0 < N_4(\lambda) \lesssim \lambda^2 \ln\ln\lambda $ and $\limsup_{\lambda \to \infty} N_4(\lambda) / (\lambda \ln \ln \lambda) > 0$.
        \smallskip
        \item $\mathcal{N}_d(\lambda) \sim_d \lambda^{d-2}$ for all $d \ge 5$.
    \end{itemize}
\end{theorem}

Next, we recall upper bounds on the number of integer points on spherical caps.
These upper bounds easily follow from the estimates obtained by Bourgain and Rudnick \cite[Appendix A]{BourgainRudnick2012}.

\begin{lemma}[Counting lattice points in caps] \label{lemma:spherical-cap}
    For $d \geq 3$ and $r \ge 1$, we have
    \begin{itemize}
        \item If $d\ge 3$, then $N_d(\lambda,r) \lesssim_\epsilon \lambda^\epsilon r^{d-2}$ for $\epsilon > 0$.
        \smallskip
        \item If $d \ge 5$, then $N_d(\lambda,r) \lesssim_{d,\epsilon} \lambda^{-1} r^{d-1} + \lambda^\epsilon r^{d-3+\epsilon}$ for $\epsilon > 0$.
    \end{itemize}
\end{lemma}

Considering now lattice points on spheres which are possibly lower-dimensional and not centered at zero, we record the following lemma, which only incurs a sub-polynomial loss compared to the best possible estimate.

\begin{lemma}[Counting lattice points on lower-dimensional spheres \cite{HuangZhang}, Lemma 4] \label{lemma:num-theo-sphere}
    If $2 \le d \le n$ and if $S^{d-1}_\lambda$ is a $(d-1)$-sphere embedded in $\R^n$ with radius $\lambda$, then for any $\epsilon>0$
    \begin{equation*}
        \# (S^{d-1}_\lambda \cap \Z^n) \lesssim_{d,\epsilon} \lambda^{d-2 + \epsilon}.
    \end{equation*}
\end{lemma}

The following lemma, originally stated by Connes \cite{Connes1976coefficients} for the case $d=n$, remains valid for any $(d-1)$-sphere embedded in $\R^n$.
This is because the proof relies solely on volume estimates of spherical caps, which are invariant under translation.

\begin{lemma}[Clustering for lattice points on spheres \cite{Connes1976coefficients}, Lemma 1.] \label{lemma:Connes}
    If $2 \le d \le n$ and if $S^{d-1}_\lambda$ is a $(d-1)$-sphere embedded in $\R^n$ with radius $\lambda$, then one may cover $S^{d-1}_\lambda \cap \Z^n$ by its subsets $(\Omega_\alpha)_\alpha$ where each $\Omega_\alpha$ lives in an affine subspace of dimension $\le d-1$ and, if $\alpha \ne \beta$, then
    \begin{equation}
        \dist(\Omega_\alpha,\Omega_\beta) \gtrsim_d \lambda^{2/(d+1)!}.
    \end{equation}
\end{lemma}

Connes constructed this partition by viewing $\mathcal{S}^{d-1}_\lambda$ as a graph: two distinct points $ p, q \in \mathcal{S}^{d-1}_\lambda $ are connected by an edge if $|p - q| \lesssim_d \lambda^{2/(d+1)!}$.
The clusters $(\Omega_\alpha)_\alpha$ are then defined as the connected components of this graph.
Particularly, when $d=2$, this lemma, which was first stated by Jarník \cite{Jarnik1926gitterpunkte} (see also \cite{CillerueloCordoba} for an improvement), implies
\begin{equation}
\label{eq::jarnik}
    \#\Omega_\alpha \le 2,
    \quad \diam \Omega_\alpha \lesssim \lambda^{1/3}, 
    \quad \forall \alpha.
\end{equation}

\subsection{A result on Fourier transforms of measures}

\begin{lemma} \label{lemma:delta-measure}
    If $\mu \in \mathcal{P} (\T^d)$ satisfies    $\lim_{k \to \infty} \widehat{\mu}_k = 0$, then $ \sup_{k \neq 0} |\widehat{\mu}_k| < 1$.
\end{lemma}
\begin{proof}
    We argue by contradiction and assume that $|\widehat{\mu}_k| = 1$ for some $k \in \Z^d \setminus \{0\}$.
    Then $e^{2\pi i k\cdot x} = e^{i\theta}$ for some $\theta \in \R$ and all $x \in \supp \mu$.
    Hence $e^{2\pi i Nk\cdot x} = e^{iN\theta}$ for all $N \in \NN \setminus 0$.
    Therefore $\widehat{\mu}_{Nk} = e^{iN\theta}$ and thus $|\widehat{\mu}_{Nk}|=1$.
    This contradicts the hypotheses of the lemma.
\end{proof}

\section{The case of dimension 2} \label{sec:dimension2}

It is instructive to consider first the case of the two-dimensional torus. The following theorem combines the results of Zygmund \cite{Zygmund1972} and Bourgain and Rudnick \cite{BourgainRudnick2009restriction}.
The proofs are short and elegant, and they are a source of inspiration throughout the present paper.

\begin{theorem}
\label{thm::ZBR} 
    When $d=2$, the following statements hold:
    \begin{enumerate}[label=(\roman*)]
        \item The trace inequality and the semiclassical observability inequality hold true if
        \begin{equation}
        \label{eq::condition-fourier-decay-2d}
            |\widehat{\mu}_k| \lesssim |k|^{-\epsilon}, \quad \forall k \in \Z^2.
        \end{equation}
        \item The trace inequality and the observability inequality hold true if 
        \begin{equation}
            \dd \mu = f \dd x,\quad f \in L^2(\T^2).
        \end{equation}
    \end{enumerate}
\end{theorem}
\begin{proof}
    For $u(x) = \sum_{k \in \mathcal{S}^{d-1}_\lambda} \widehat{u}_k e^{2\pi i k\cdot x}$ in $E_\lambda$, write
    \begin{equation}
        \int_{\T^2} |u|^2 \dd \mu  = \sum_{k,\ell \in \mathcal{S}^1_\lambda} \widehat{\mu}_{k-\ell} \widehat{u}_k \overline{\widehat{u}_\ell}.
    \end{equation}

    \noindent\textbf{The trace inequality:}
    If $\dd \mu = f \dd x$ with $f$ in $L^2$, then $(\widehat{\mu}_k)_{k \in \Z^2} \in \ell^2$ by Parseval's theorem. Furthermore, any $\xi \in \Z^2$ can be represented as $k-\ell$ with $k,\ell \in \mathcal{S}^1_\lambda$ in at most two different ways. 
    With the Cauchy--Schwarz inequality, this gives 
    \begin{equation*}
        \biggl| \sum_{k,\ell \in \mathcal{S}^1_\lambda} \widehat{\mu}_{k-\ell} \widehat{u}_k \overline{\widehat{u}_\ell} \biggr| \leq \biggl( \sum_{k,\ell \in \mathcal{S}^1_\lambda} |\widehat{\mu}_{k-\ell}|^2 \biggr)^{1/2}  \biggl( \sum_{k,\ell \in \mathcal{S}^1_\lambda} |\widehat{u}_k|^2 |\widehat{u}_\ell|^2 \biggr)^{1/2} \leq \sqrt{2} \| f \|_{L^2} \| u \|_{L^2}^2.
    \end{equation*}

    If $|\widehat{\mu}_k| \lesssim |k|^{-\epsilon}$, we use Jarnik's lemma (Lemma~\ref{lemma:Connes} or, more precisely, \eqref{eq::jarnik}) which gives a splitting of $\mathcal{S}^1_\lambda$ into sets $(\Omega_\alpha)$ such that $\# \Omega_\alpha \leq 2$ and $\dist(\Omega_\alpha,\Omega_\beta) \gtrsim \lambda^{1/3}$ for $\alpha \neq \beta$. 
    Write
    \begin{equation}
    \label{eq::OBS-decomp-jarnik}
        \sum_{k,\ell \in \mathcal{S}^1_\lambda} \widehat{\mu}_{k-\ell} \widehat{u}_k \overline{\widehat{u}_\ell} 
        = \sum_\alpha \sum_{k,\ell \in \Omega_\alpha} \widehat{\mu}_{k-\ell} \widehat{u}_k \overline{\widehat{u}_\ell} 
        + \sum_{\alpha \neq \beta} \sum_{k \in \Omega_\alpha} \sum_{\ell \in \Omega_\beta} \widehat{\mu}_{k-\ell} \widehat{u}_k \overline{\widehat{u}_\ell}.
    \end{equation}
    By the Cauchy--Schwarz inequality and the bound $\sharp \mathcal{S}^1_\lambda = N_1(\lambda) \lesssim_\epsilon \lambda^{\epsilon/4}$ for any $\epsilon>0$, we deduce
    \begin{align*}
        \biggl| \sum_{k,\ell \in \mathcal{S}^1_\lambda} \widehat{\mu}_{k-\ell} \widehat{u}_k \overline{\widehat{u}_\ell} \biggr| 
        & \leq \biggl|\sum_\alpha \sum_{k,\ell \in \Omega_\alpha} \widehat{\mu}_{k-\ell} \widehat{u}_k \overline{\widehat{u}_\ell} \biggr| 
        + \biggl| \sum_{\alpha \neq \beta} \sum_{k \in \Omega_\alpha} \sum_{\ell \in \Omega_\beta} \widehat{\mu}_{k-\ell} \widehat{u}_k \overline{\widehat{u}_\ell} \biggr| \\
        & \lesssim_\epsilon \sum_\alpha \sum_{k,\ell \in \Omega_\alpha} | \widehat{u}_k \widehat{u}_\ell | + \lambda^{-\epsilon/3} \sum_{k,\ell \in \mathcal{S}^1_\lambda } |\widehat{u}_k \widehat{u}_\ell| 
        \\
        & \lesssim \sum_\alpha \sup_\alpha \#\Omega_\alpha \sum_{k \in \Omega_\alpha} |\widehat{u}_k|^2 + \lambda^{-\epsilon/3} N_1(\lambda) \sum_{k \in \mathcal{S}^1_\lambda} |\widehat{u}_k|^2 \lesssim \| u \|_{L^2}^2.
    \end{align*}

    \noindent\textbf{The observability inequality:}
    Similarly, we reuse \eqref{eq::OBS-decomp-jarnik} and estimate
    \begin{equation*}
        \biggl| \sum_{\alpha \neq \beta} \sum_{k \in \Omega_\alpha} \sum_{\ell \in \Omega_\beta} \widehat{\mu}_{k-\ell} \widehat{u}_k \overline{\widehat{u}_\ell} \biggr|
        \lesssim \begin{cases}
        \sum_{|k| \gtrsim \lambda^{-1/3}} |\widehat{\mu}_k|^2, & \mu \in L^2; \\
        C_\epsilon \lambda^{\epsilon/4} \lambda^{ - \epsilon/3}, & |\widehat{\mu}_k| \lesssim |k|^{-\epsilon}.
        \end{cases}
    \end{equation*}
    The important point is that the right-hand side goes to zero as $\lambda \to \infty$. 
    There remains the first sum in the above equation.
    We analyze the following two cases separately:
    \begin{itemize}
        \item If $\Omega_\alpha = \{p\}$ is a singleton, then clearly
        \begin{equation*}
            \sum_\alpha \sum_{k,\ell \in \Omega_\alpha} \widehat{\mu}_{k-\ell} \widehat{u}_k \overline{\widehat{u}_\ell}
            = |\widehat{u}_p|^2.
        \end{equation*}
        \item If $\Omega_\alpha = \{p,q\}$ is a doubleton, then
        \begin{align*}
            \sum_\alpha \sum_{k,\ell \in \Omega_\alpha} \widehat{\mu}_{k-\ell} \widehat{u}_k \overline{\widehat{u}_\ell}
            & = |\widehat{u}_p|^2 + |\widehat{u}_q|^2 + \widehat{\mu}_{p-q} (\widehat{u}_p \overline{\widehat{u}_q} + \widehat{u}_q \overline{\widehat{u}_p}) \\
            & \ge (1-|\widehat{\mu}_{p-q}|) (|\widehat{u}_p|^2 + |\widehat{u}_q|^2).
        \end{align*}
        This estimate is uniform among all such $\Omega_\alpha$ since by Lemma \ref{lemma:delta-measure}, we have
        \begin{equation*}
            1-|\widehat{\mu}_{p-q}|
            \ge 1 - \sup_{k \in \Z^2 \setminus 0} |\widehat{\mu}_k| > 0.
        \end{equation*}
    \end{itemize}

    This proves the semiclassical observability for both cases $(i)$ and $(ii)$.
    In the case $(ii)$, there remains to prove the observability for eigenfunctions with bounded eigenvalues.
    This is immediate, since trigonometric polynomials cannot vanish on sets of positive Lebesgue measure.
\end{proof}

\begin{remark}
    A number of remarks and consequences are of interest.
    \begin{itemize} 
        \item The condition \eqref{eq::condition-fourier-decay-2d} is satisfied if $\mu$ is given by a smooth density on a smooth curve with non-vanishing curvature --- this was the original motivation of Bourgain and Rudnick.
        Since the zero set of eigenfunctions of the Laplacian can include curves of nonzero curvature, this example shows that the observability inequality can be satisfied in the limit $\lambda \to \infty$ although it is not satisfied for small values of $\lambda$.
        \item If $\epsilon$ is the optimal polynomial decay rate of the Fourier transform of the measure in \eqref{eq::condition-fourier-decay-2d}, then $2\epsilon$ is called the Fourier dimension of the measure, see Mattila \cite{Mattila2015} for many examples and an introduction to this theory.
        \item We prove that the trace inequality constant $\mathfrak{C}_\mu$ is bounded by $\sup_k | \widehat{\mu}_k| |k|^\epsilon$ and $\| \mu \|_{L^2}$ in assertions $(i)$ and $(ii)$ respectively. 
        The semiclassical observability constant $\mathfrak{c}_\mu^{\operatorname{sc}}$ is bounded by $(1 - \sup_{k \neq 0} |\widehat{\mu}_k|)^{-1}$ in both cases. 
        \item The assertion $(ii)$ implies observability for any probability measure of the form $\dd \mu = f \dd x$ with $f \in L^1$. Indeed, let $f_M = \min\{f,M\}$ with $M>0$ sufficiently large, then $f_{ M} \in L^2$ and $f \geq f_{M} $.
        The observability inequality for $f_M$ implies that for $f$.
    \end{itemize}
\end{remark}

\begin{remark}
    The above theorem can be extended rather straightforwardly to the following cases
    \begin{itemize}
        \item Superficial measures supported on non-smooth curves, $\dd \mu = \dd \sigma_\Gamma$. One may assume either $\Gamma$ is twice continuously differentiable with non-vanishing curvature, or $\Gamma$ is convex and rely on \cite{Chakhiev} to obtain decay for the Fourier transform.
        \item Functions in Sobolev spaces $W^{\epsilon,p}$ with $\epsilon > 0$ and $p \ge 1$ (clearly it gives new results only when $p \in [1,2)$).
        To show this, it suffices to use Jarnik's lemma as above, and notice that the Sobolev regularity gives polynomial gain in $\lambda$ while the number of lattice points on the circle is subpolynomial.
    \end{itemize}
    Motivated these results, it is natural to investigate whether the trace and observability inequalities extend to broader classes of measures and sets. In particular, the following cases are of interest and open to the best of our knowledge:
    \begin{itemize}
        \item \textbf{Trace inequality:} Densities belonging to \( L^p(\T^2) \) for \( 1 < p < 2 \).
        \item \textbf{Trace inequality:} Densities in \( L^p(\Sigma) \), where \( \Sigma \subset \T^2 \) is a smooth curve.
        \item \textbf{Observability inequality:} Subsets of a smooth curve \( \Sigma \subset \T^2 \) with non-vanishing curvature, provided these subsets have positive measure.
        \item \textbf{Trace and observability inequalities:} Measures supported on sets with Hausdorff dimension strictly greater than 1.
    \end{itemize}
\end{remark}

\section{Necessary conditions for the  trace inequality} \label{sec:trace-nec}

We will show in this section that probability measures which are  $(d-2)$-regular are critical to have the trace inequality. The first examples which come to mind are 
\begin{itemize}
    \item $\dd \mu = f \dd x$ where $f$ is smooth away from $x=0$ and $ f(x) \sim |x|^{-2}$ near $x =0$.
    \item $\dd \mu = \phi \dd \sigma_\Gamma$ where  $\Gamma$ is a smooth $(d-2)$-dimensional manifold and $\phi \in C^\infty(\Gamma)$.
\end{itemize}

Our idea is to use concentration properties of the so-called Bourgain eigenfunctions. Precisely, these are eigenfunctions of the form
\begin{equation}
\label{eq::eigenfunction-bourgain}
    \varphi_{\lambda,x_0} (x) = \sum_{k \in \mathcal{S}^{d-1}_\lambda} e^{2\pi i k \cdot (x- x_0)},
    \quad
    \lambda \in \sqrt{\NN},
    \ 
    x_0 \in \T^d.
\end{equation}
Clearly $\|\varphi_{\lambda,x_0}\|_{L^2} = \sqrt{N_d(\lambda)}$ and $\varphi_{\lambda,x_0}(x_0) = N_d(\lambda)$.
Notice that, there exists $c_0 > 0$ such that, if $x \in B_r(x_0)$ with $r\lambda \le c_0$, then $\cos\bigl(2\pi k\cdot(x-x_0)\bigr) \ge \frac{1}{2}$ for all $k \in \mathcal{S}^{d-1}_\lambda$.
For such $x$,
\begin{equation}
\label{eq::bourgain-planck-concentration}
    |\varphi_{\lambda,x_0}(x)| 
    \ge |\Re \varphi_{\lambda,x_0}(x)|
    \ge \sum_{k \in \mathcal{S}^{d-1}_\lambda} \cos\bigl(2\pi k\cdot(x-x_0)\bigr)
    \ge \frac{1}{2} N_d(\lambda).
\end{equation}

\begin{proposition} 
    \label{propnecessary3}
    If $d \geq 3$ and if the trace inequality holds for $ \mu \in \mathcal{P} (\T^d)$, then $\mu$ is upper $(d-2)$-regular, i.e., $\mu \bigl(B_r(x_0)\bigr) \lesssim r^{d-2}$ for all $r >0$ and $x_0 \in \T^d$.
    Consequently,
    \begin{equation}
        \dim_{H} (\supp \mu) \geq d-2.
    \end{equation}
\end{proposition}
\begin{proof}
    Choosing $c_0 > 0$ as above.
    By Lemma~\ref{lem::lattice-number-lower-bound}, when $r>0$ is sufficiently small, there exists $\lambda \in [ c_0 r^{-1}, 2 c_0 r^{-1} ] \cap \sqrt{\NN}$ such that 
    $N_d(\lambda) \gtrsim_d \lambda^{d-2}$.
    For such $\lambda$ and $x_0 \in \T^d$, define $\varphi_{\lambda,x_0} \in E_\lambda$ as in \eqref{eq::eigenfunction-bourgain}.
    By \eqref{eq::bourgain-planck-concentration} and the trace inequality for $\mu$, there holds
    \begin{equation}
    \label{eq::frostman-estimate}
    \begin{split}
        \mu \bigl(B_{r} (x_0)\bigr) 
        = \int_{B_r (x_0)} \dd \mu 
        & \lesssim  \frac{1}{N_d(\lambda)^2} \int_{\T^d} |\varphi_{\lambda,x_0} |^2 \dd\mu  \\
        & \lesssim_\mu \frac{1}{N_d(\lambda)^2} \int_{\T^d} |\varphi_{\lambda,x_0}|^2 \dd x 
        =  \frac{1}{N_d(\lambda)} \lesssim  \frac{1}{\lambda^{d-2}}
        \lesssim_d r^{d-2}.
    \end{split}
    \end{equation}
    Applying Frostman's Lemma we conclude the proof.
\end{proof}

When $d=2$, the above proposition holds true, but the statement ``\emph{$\mu$ is upper $0$-regular}'' is trivial.
In the following proposition, we obtain a non-trivial estimate.

\begin{proposition}
\label{prop::trace-necessary-sufficient-d=2}
    Let $d=2$. 
    {There exists $C > 0$ such that,}
    if the trace inequality holds for $\mu \in \mathcal{P} (\T^d)$, then for any $r \in (0,1)$ and $x_0 \in \T^2$, we have
    \begin{equation}
        \mu(B_r(x_0)) \lesssim r^{\frac{C}{\ln \ln (1/r)}}.
    \end{equation}
\end{proposition}
\begin{proof} 
    We follow a line of argument already appearing in \cite{GermainMoyanoZhu2024vanishing}.
    Let $P_n$ be the product of all prime numbers which are $\equiv 1 \pmod{4}$ and $\le n$. By \cite{McCurley-Kevin1984prime} and \cite{McCurley-Kevin1984-Lfunction}, we have $\ln P_n \sim n$. For $r > 0$, choose $n$ such that $\sqrt{P_{n-1}} \le 1/r\le \sqrt{P_n}$, then $n \sim \ln (1/r)$.
    Then there holds
    \begin{equation}
    \label{eq::primorial-upper-bound}
        \sqrt{P_n} \le \frac{1}{r} \frac{\sqrt{P_n}}{\sqrt{P_{n-1}}} \le \frac{\sqrt{n}}{r} \lesssim \frac{\sqrt{\ln(1/r)}}{r}.
    \end{equation}
    Next, let $\pi_n$ be the number of prime numbers $\le n$ which are $\equiv 1 \pmod{4}$ and recall that, by the prime number theorem for arithmetic progressions,
    $\pi_n \sim n/\ln n$.
    
    Let $r>0$ be sufficiently small and let $\lambda = \sqrt{P_n}$.
    By \eqref{eq::primorial-upper-bound}, we have $1/r \le \lambda \lesssim \sqrt{\ln(1/r)}/r$.
    Let $\rho = r/\sqrt{\ln(1/r)}$.
    We claim that,  when $r \le e^{-1}$ (which is always the case since $r$ is sufficiently small), there holds $\rho \le r \le \rho \sqrt{\ln(1/\rho)}$.
    Indeed, by definition $\rho \le r/\sqrt{\ln e} = r$ and $r = \rho \sqrt{\ln(1/r)} \le \rho \sqrt{\ln(1/\rho)}$.
    Combining these estimates gives
    \begin{equation*}
        \frac{1}{\rho \sqrt{\ln(1/\rho)}} \le \lambda \lesssim \frac{1}{\rho}.
    \end{equation*}
    Therefore, by Jacobi (see \cite{Grosswald}) and the definition of the primorial $P_n$, letting $d_j(n)$ be the number factors of $n$ that are $\equiv j \pmod{4}$, then for some $C,C',C''>0$,
    \begin{equation}
    \label{eq::lattice-number-primorial-est-d=2}
    \begin{split}
        N_2(\lambda) 
        & = 4 \bigl(d_1(P_n) - d_3(P_n)\bigr)
        = 4 d_1(P_n) \\
        & = 4 \times 2^{\pi_n}
        \ge 4 \times 2^{C n/\ln n}
        \gtrsim \lambda^{C'/\ln\ln\lambda}
        \ge \rho^{-C''/\ln\ln(1/\rho)}.
    \end{split}
    \end{equation}
    
    For $x_0 \in \T^2$, define $\varphi_{\lambda,x_0}$ as in \eqref{eq::eigenfunction-bourgain}.
    Fix a sufficiently small $\delta>0$ (independent of $r$ or $\lambda$) such that, by \eqref{eq::bourgain-planck-concentration}, if $x \in B_{\delta\rho}(x_0)$, then  $|\phi_{\lambda,x_0}(x)| \gtrsim N_2(\lambda)$.
    Arguing as in \eqref{eq::frostman-estimate} gives
    \begin{equation*}
        \mu \bigl(B_{\delta\rho} (x_0)\bigr) 
        =  \frac{1}{N_2(\lambda)}
        \lesssim \rho^{-C''/\ln\ln(1/\rho)}.
        \qedhere
    \end{equation*}
\end{proof}

When $d=3,4$, the number $N_d(\lambda)$ exceeds $\lambda^{d-2}$ by an arbitrarily large factor.
This causes the failure of the trace inequality for these critical measures, as is stated in the following theorem.

\begin{proposition} \label{prop:d3d4}
    When $d=3,4$ the  trace inequality does not hold for $\mu \in \mathcal{P} (\T^d)$ if for all $\lambda \in \sqrt{\NN}$, there exists $x_\lambda \in \T^d$ such that
    \begin{equation}
    \label{eq::sequential-lower-(d-2)-regular}
        \mu\bigl(B_{\lambda^{-1}}(x_\lambda)\bigr) \gtrsim \lambda^{2-d}.
    \end{equation}
\end{proposition}

\begin{remark}
    It is known that (Theorem~\ref{thm::cardinality-lattice-sphere}), for $d=3,4$, there exists $(\lambda_n)_{n \ge 0} \in \sqrt{\NN}$ such that
    \begin{equation}
    \label{eq::number-lattice-limsup-d=3-4}
        \lim_{n \to \infty} \frac{N_d(\lambda_n)}{\lambda_n^{d-2} \ln \ln \lambda_n} > 0.
    \end{equation}
    From the proof below, one immediately sees that the condition \eqref{eq::sequential-lower-(d-2)-regular} can be relaxed to
    \begin{equation}
        \limsup_{n \to \infty} \mu\bigl(B_{\lambda_n^{-1}}(x_{\lambda_n})\bigr)   \lambda_n^{d-2} \ln\ln \lambda_n   = \infty.
    \end{equation}
\end{remark}

\begin{proof}
    Let $\lambda_n$ satisfy \eqref{eq::number-lattice-limsup-d=3-4} and define $\varphi_n = \varphi_{\lambda_n,x_{\lambda_n}}$ as in \eqref{eq::eigenfunction-bourgain}.
    By \eqref{eq::bourgain-planck-concentration}, there exists $\delta \in (0,1)$ such that that $|\varphi_n(x)| \gtrsim N_d(\lambda_n)$ for all $x\in B_{\delta \lambda_n^{-1}}(x_{\lambda_n})$.
    Therefore, as $n \to \infty$,
    \begin{equation*}
        \frac{1}{N_d(\lambda_n)} \int_{\T^d} |\varphi_n|^2 \dd \mu  \gtrsim
        N_d(\lambda_n) \mu(B_{\delta \lambda_n^{-1}}(x_{\lambda_n} ))
        \gtrsim_\delta \ln \ln \lambda_n \to \infty. \qedhere
    \end{equation*}
\end{proof}

We conclude this section with the following proposition, which excludes the possibility of observability on linear subspaces of codimension 2. Unlike previous examples, whose constructions relied on eigenfunctions concentrating at a point, the proof here involves a sequence of eigenfunctions concentrating along the entire subspace. For simplicity, the subspace considered is aligned with the coordinate axes; however, the argument should extend to arbitrary rational linear subspaces of codimension 2.

\begin{proposition}
    For $d \geq 3$, the trace inequality fails for $\mu \in \mathcal{P}(\T^d)$ if
    \begin{equation*}
        \supp \mu \subset S = \{x \in \T^d: x_1=x_2 =0 \}.
    \end{equation*}
\end{proposition}

\begin{proof}
    For $n \ge 0$, let $\lambda_n = \sqrt{P_n}$ where the primorial $P_n$ is defined in the proof of Proposition~\ref{prop::trace-necessary-sufficient-d=2}.
    By \eqref{eq::lattice-number-primorial-est-d=2}, we have $\lim_{n \to \infty} N_2(\lambda_n) = \infty$.
    For $x = (x_1,\ldots,x_d) \in \T^d$, put
    \begin{equation*}
        g_n(x) = \sum_{(k_1,k_2) \in \mathcal{S}^1_\lambda} e^{2\pi i (k_1 x_1 + k_2 x_2)}.
    \end{equation*}
    Then $\| g_n \|_{L^2} = \sqrt{N_2(\lambda_n)}$ and $g_n(x) = N_2(\lambda_n)$ for $x \in S$.
    We conclude with
    \begin{equation*}
        \lim_{n \to \infty} \frac{\int_{\T^{d}} |g_n(x)|^2  \dd  \mu }{\int_{\T^d} |g_n(x)|^2 \dd x} 
        = \lim_{n \to \infty} N_2(\lambda_n) = \infty. \qedhere
    \end{equation*}
\end{proof}

\section{Necessary conditions for the  observability inequality} \label{sec:observability-nec}

Before considering general measures, see Proposition \ref{prop:no-observability-d-2}, we start with the case of superficial measures on submanifolds.

\begin{proposition}[Necessary conditions for observability on local manifolds]\label{local}
    Let $d\geq 3$ and let $S \subset \T^d$ be a compact, non-trivial and smooth submanifold of integer dimension $\beta \le d-2$.
    There exists an $\epsilon >0$ such that for any $x_0 \in S$, the observability inequality does not hold for the probability measure $\mu_{x_0,\epsilon} \in \mathcal{P}(\T^d)$ defined by
    \begin{equation*}
        \dd \mu_{x_0, \epsilon} = \frac{\mathbf{1}_{B_\epsilon(x_0) \cap S}}{\mathcal{H}^\beta (B_\epsilon(x_0) \cap S)}  \dd \mathcal{H}^\beta. 
    \end{equation*}
    Precisely, for any $x_0\in S$, there exists an increasing sequence $(\lambda_n)_{n \ge 0}$ in $\sqrt{\NN}$ and a sequence of $L^2$-normalized eigenfunctions $(u_n)_{n\ge 0}$ with $u_n \in E_{\lambda_n}$ such that 
    \begin{equation}
    \label{eq::lim-no-obs-local}
        \lim_{n \to \infty}
        \int_{\T^d} |u_n|^2 \dd \mu_{x_0, \epsilon} = 0.
    \end{equation}
\end{proposition}

\begin{proposition}[Necessary conditions for observability on global manifolds]\label{global}
    Let $d \ge 3$ and $S \subset \T^d$ be a compact, non-trivial, smooth submanifold of integer dimension $\beta$, where $\beta \le d-2$ for $d=3,4$ and $\beta \leq d-3$ for $d\ge 5$.
    Then the  observability inequality does not hold for the probability measure $\mu \in \mathcal{P}(\T^d)$ defined by
    \begin{equation*}
        \dd \mu = \frac{\mathbf{1}_{S}}{\mathcal{H}^\beta (S)}  \dd \mathcal{H}^\beta. 
    \end{equation*}
    Precisely, there exists an increasing sequence $(\lambda_n)_{n \ge 0}$ in $\sqrt{\NN}$ and a sequence of $L^2$-normalized eigenfunctions $(u_n)_{n\ge 0}$ with $u_n \in E_{\lambda_n}$ such that 
    \begin{equation}
    \label{eq::lim-no-obs-global}
        \lim_{n \to \infty}
        \int_{\T^d} |u_n|^2 \dd \mu = 0.
    \end{equation}
\end{proposition}

\begin{remark}
    From our proofs, the convergence rates of \eqref{eq::lim-no-obs-local} and \eqref{eq::lim-no-obs-global} is faster than any negative power of $\lambda_n$.
    Consequencely, no observation estimates can hold even with loss of derivatives.
\end{remark}

\begin{proof}[Proof of Proposition~\ref{local}] 
    Since $S$ is compact, there exists $\epsilon_0 > 0$ such that for any $x_0 \in S$, the submanifold $S$ is parametrized as $S_{x_0} = \{(x',\varphi(x')) : x' \in \epsilon_0 B\}$ near $x_0$, where $B \subset \R^\beta$ is the unit ball centered at the origin, and $\varphi: \R^\beta \to \R^{d-\beta}$ a smooth and $x_0$-dependent function.
    
    For $k \in \mathcal{S}^{d-1}_\lambda$, the trace of the eigenfunction $e^{2\pi i k\cdot x}$ on $S_{x_0}$ is $e^{2\pi i k \cdot (x',\varphi(x'))}$.
    Next, let $\chi \in C_c^\infty(\R^\beta)$ be such that it equals to one in the unit ball $B$ and is supported in $2B$.
    When $0 < \epsilon < \frac{1}{2} \epsilon_0$, for $a_k \in \mathbb{C}$ where $k \in \mathbb{S}^{d-1}_\lambda$, one defines $g_\lambda \in C_c^\infty(\R^\beta)$ as
    \begin{equation}
        g_\lambda(x')
        = \chi \Bigl(\frac {x'} {\epsilon}\Bigr) \sum_{k \in \mathcal{S}^{d-1}_{\lambda}}  a_k e^{2\pi i k \cdot (x',\varphi(x'))}.
    \end{equation}
    Since $\supp g_\lambda \subset 2\epsilon B$, we may naturally regard it as a $4\epsilon \Z^\beta$-periodic function, so that it lives henceforth in $\T^\beta_{4\epsilon} = \R^\beta / 4\epsilon \Z^\beta$.
    Expanding it in Fourier series, one writes
    \begin{equation}
        g_\lambda(x') 
        = (4\epsilon)^{-\beta} \sum_{\ell \in \Z^\beta} e^{2\pi i \frac{\ell}{4\epsilon} \cdot x'} \widehat{g}_\lambda(\ell),
        \quad
        \widehat{g}_\lambda(\ell) = \sum_{k \in \mathcal{S}^{d-1}_\lambda} a_k A^\epsilon_\lambda(k,\ell),
    \end{equation}
    where the linear coefficients are
    \begin{equation}
         A^\epsilon_\lambda(k,\ell)
         = \int_{\T^\beta_{4\epsilon}} \chi \Bigl(\frac {x'} {\epsilon}\Bigr) e^{-2\pi i \frac{\ell}{4\epsilon} \cdot x'} e^{2\pi i k \cdot (x',\varphi(x'))} \dd x'.
    \end{equation}
    
    We will fix parameters $0 < \epsilon < \eta$ depending solely on $d$ and $\beta$.
    Then we will find an increasing sequence of $\lambda_n$ and a sequence of nontrivial constants $(a^{(n)}_k:k \in \mathcal{S}^{d-1}_{\lambda_n})$ such that the Fourier coefficients $\widehat{g}_{\lambda_n}(\ell)$ vanish when $|\ell| \le \eta \lambda_n$.
    This leads to solving the linear systems:
    \begin{equation}\label{null}
        \widehat{g}_{\lambda_n}(\ell) = \sum_{k \in \mathcal{S}^{d-1}_{\lambda_n}} a_k^{(n)} A^\epsilon_{\lambda_n}(k,\ell) = 0, \quad \ell \in \Z^\beta, \ |\ell| \le \eta \lambda_n.
    \end{equation}
    The number of equations is $\lesssim_\beta \eta^\beta \lambda^\beta$ and the number unknowns (which are $(a^{(n)}_k:k \in \mathcal{S}^{d-1}_{\lambda_n})$) is $N_d(\lambda_n)$.
    By Lemma~\ref{lem::lattice-number-lower-bound}, we may choose $\lambda_n$ such that $N_d(\lambda_n) \gtrsim_d \lambda_n^{d-2}$.
    Since $\beta \le d-2$, if $\eta$ is small and $n$ is large, the system \eqref{null} admits a nontrivial solution $(a_k^{(n)}:k \in \mathcal{S}^{d-1}_{\lambda_n})$, which we may further assume to be normalized in $\ell^2(\mathcal{S}^{d-1}_{\lambda_n})$.
    
    Fix these solutions.
    By Lemma~\ref{nonstat} below, since $\epsilon < \eta$, if $|\ell| > \eta \lambda_n$, then for all $N \ge 0$,
    \begin{equation*}
        |\widehat{A}^\epsilon_{\lambda_n}(k,\ell)|
        \lesssim_{d,\beta,\varphi,\epsilon,\eta,N} \ell^{-3N}.
    \end{equation*}
    For $N \geq d$, by the Cauchy--Schwartz inequality, this leads to
    \begin{equation*}
        |\widehat{g}_{\lambda_n}(\ell)| \le \Bigl( \sum_{k \in \mathcal{S}^{d-1}_{\lambda_n}} |A_{\lambda_n}^\epsilon(k,\ell)|^2\Bigr)^{1/2}
        \lesssim_{d,\beta,\varphi,\epsilon,\eta,N} |\ell|^{-3N} N_d(\lambda_n)
        \lesssim_{d,\eta} |\ell|^{-2N}.
    \end{equation*}
    Recalling that $\widehat{g}_{\lambda_n} (\ell) =0$ when $|\ell| \leq \eta \lambda_n $,
    by Parseval's identity we deduce
    \begin{equation*}
        \| g_{\lambda_n} \|_{L^2 (\T^\beta_{4\epsilon})}^2
        \sim_{\beta,\epsilon}
        \sum_{|\ell| \geq \eta \lambda_n } |\widehat{g}_{\lambda_n} (\ell)|^2
        \lesssim_{d,\beta,\varphi,\epsilon,\eta,N} \lambda_n^{- 4N + \beta}
        \le \lambda_n^{-N}.
    \end{equation*}
    
    Finally, letting $u_n(x) = \sum_{k \in \mathcal{S}^{d-1}_{\lambda_n}} a_k^{(n)} e^{2\pi i k\cdot x} \in E_{\lambda_n}$, then
    \begin{equation*}
        \| u_n \|_{L^2(\mu_{x_0,\epsilon})}^2
        \lesssim_{\beta,\epsilon} \| g_{\lambda_n} \|_{L^2 (\T^\beta_{4\epsilon})}^2 \lesssim_{d,\beta,\varphi,\epsilon,\eta,N} \lambda_n^{-N}.
        \qedhere
    \end{equation*}
\end{proof}

\begin{lemma}\label{nonstat}
   If $0 < \epsilon < \eta$ and $|\ell| > \eta\lambda$, then for all $N \in \NN$, there holds
   \begin{equation}
       |A_\lambda^\epsilon(k,\ell)| \lesssim_{d,\beta,\varphi,\epsilon,\eta,N} |\ell|^{-N}.
   \end{equation}
\end{lemma}

\begin{proof}
    Write $k = (k',k'') \in \Z^\beta \times \Z^{d-\beta}$ and let $\psi^\epsilon_{k,\ell}(x') = (k' -\frac{\ell}{4\epsilon}) \cdot x' +  k'' \cdot \varphi(x')$.
    Then
    \begin{equation*}
        A_\lambda^\epsilon(k,\ell) = \int_{\T^\beta_{4\epsilon}} \chi\Bigl(\frac{x'}{\epsilon}\Bigr) e^{2\pi i \psi^\epsilon_{k,\ell}(x')} \dd x'.
    \end{equation*}
    Since $|k'| \le |k| = \lambda < \frac{|\ell|}{\eta} $, to conclude, we integrate by part and use the estimate:
    \begin{equation} \label{non-degeneracy-cond}
        |\nabla_{x'} \psi^\epsilon_{k,\ell} (x')|  = \Bigl|\Bigl(k' - \frac \ell \epsilon, k'' \cdot \nabla_{x'} \varphi\Bigr) \Bigr|
        \geq \frac{|\ell|} \epsilon - |k'|
        \geq |\ell| \Bigl(\frac{1}{\epsilon} - \frac{1}{\eta}\Bigr).
        \qedhere
    \end{equation} 
\end{proof}

\begin{proof}[Proof of Proposition \ref{global}]
    By the compactness of $\mathcal{S}$, there exists $\epsilon_0 > 0$ and a finite open covering $S = \bigcup_{j=1}^K S_j$ where, for each $j$, there exists a smooth function $\varphi_j : \R^\beta \to \R^{d-\beta}$ such that $S_j = \{(x',\varphi_j(x') : x' \in \epsilon_0 B) \}$ under a certain system of coordinates.
    
    We proceed as in the proof of Proposition~\ref{local} and let
    \begin{equation}
        g^j_\lambda(x')
        = \chi \Bigl(\frac {x'} {\epsilon}\Bigr) \sum_{k \in \mathcal{S}^{d-1}_{\lambda}}  a_k e^{2\pi i k \cdot (x',\varphi_j(x'))}
        = (4\epsilon)^{-\beta} \sum_{\ell \in \Z^\beta} e^{2\pi i \frac{\ell}{4\epsilon} \cdot x'} \widehat{g}^j_\lambda(\ell).
    \end{equation}
    Fix an increasing sequence $(\lambda_n)_{n \ge 0}$ such that
    \begin{itemize}
        \item If $d=3,4$, then $N_d(\lambda_n) \gtrsim \lambda_n \ln \ln \lambda_n$;
        \smallskip
        \item If $d\ge 5$, then $N_d(\lambda_n) \gtrsim_d \lambda_n^{d-2}$.
    \end{itemize}
    Setting $\eta = 1$, we solve the following linear systems:
    \begin{equation}
        \widehat{g}^j_{\lambda_n}(\ell) = 0, \quad \ell \in \Z^\beta,\ |\ell| \le \lambda_n, \ 1 \le j \le K.
    \end{equation}
    Now the number of equation is $\lesssim_\beta K \lambda_n^\beta$.
    It is smaller than $N_d(\lambda_n)$, the number of the unknowns (which are $(a_k)$), under the assumption of the proposition (either $d=3,4$ and $\beta \le d-2$ or $d \ge 5$ and $\beta \le d-3$).
    Hence, for large $n$, we find an $\ell^2$ normalized solution $(a^n_k)_{k \in \mathcal{S}^{d-1}_{\lambda_n}}$ to this system.
    
    To conclude, one defines the eigenfunction $u_n \in E_{\lambda_n}$ as previously, apply the previously obtained estimates of $u_n$ on each coordinate patch, and combines these estimate:
    \begin{equation*}
    \|u_n\|_{L^2(\dd\mu)}^2
    \le \sum_{1 \le j \le K} \| \bm{1}_{S_j} u_n \|_{L^2(\dd\mu)}^2
    \lesssim_{\beta,\epsilon} \sum_{1 \le j \le K} \| g^j_{\lambda_n} \|_{L^2 (\T^\beta_{4\epsilon})}^2 
        \lesssim_{d,\beta,\varphi,\epsilon,\eta,N} K \lambda_n^{-N}.
        \qedhere
    \end{equation*}
\end{proof}

\begin{remark} 
    Propositions~\ref{local} and~\ref{global} (the latter under the requirement $\beta \leq d-3$) can be extended to a general compact Riemannian manifold $M$ of dimension $d$.
    In this generalization, eigenfunctions are replaced by quasimodes of $\Delta_M$, the Laplacian on $M$. 
    Specifically, for some $\delta > 0$, one considers functions in the space
    \begin{equation*}
        E_{\lambda,\delta} = \bigoplus_{\mu \in (\lambda^2 - \delta, \lambda^2 + \delta)} \ker(-\Delta_M - \mu).
    \end{equation*}
    Indeed, at least along a sequence $\lambda_n \to \infty$, the lower bound
    \begin{equation*}
        \dim E_{\lambda,\delta} \gtrsim_{M,\delta} \lambda^{d-2}
    \end{equation*}
    follows directly from Weyl's law:
    \begin{equation*}
        \#\{ \lambda \in \sigma(-\Delta) : \lambda \leq \Lambda \} \sim_M \Lambda^d.
    \end{equation*}
    In the preceding proofs, the only torus-specific ingredient is the explicit form of the eigenfunctions on $\T^d$, which makes it possible to precisely identify the oscillatory behavior of the traces $g_\lambda$. This structure enables the integration by parts argument in Lemma~\ref{nonstat}.
    However, on a general compact Riemannian manifold, this exact eigenfunction description can be replaced by suitable approximate representations of quasimodes (see e.g.~\cite[Chapter5]{Sogge} and~\cite[Lemma 2.3]{BurqGerardTzvetkov-2}). 
\end{remark}

We now turn to the case of general measures. The reader can refer to Section~\ref{sec:prel} for a definition of the upper Minkowski dimension.

\begin{proposition} [Necessary conditions for observability for general measures] \label{prop:no-observability-d-2} 
    If $d \geq 3$ and $\overline{\dim}_{\mathrm{M}} (\supp \mu) < d-2$, then the observability inequality does not hold true.
    Equivalently, there exists a sequence of $L^2$-normalized eigenfunction $\{u_n \}_{n}$ where $u_n \in E_{\lambda_n}$ with $\lambda_n \to \infty$, such that
    \begin{equation*}
        \lim_{n \to \infty} \int_{\T^d} |u_n|^2 \dd \mu = 0.
    \end{equation*}
\end{proposition}

We start with a preparatory lemma.

\begin{lemma}
\label{preplemma}
    For all $\lambda \ge 1$, all $\ell \in \Z^d$ and all $x_0 \in \T^d$, there exists a linear functional $\alpha_{\lambda,\ell,x_0} : C^\infty(\T^d) \to \mathbb{C}$ such that, if $u \in C^\infty (\T^d)$ is $L^2$-normalized and satisfies $\widehat{u}_k = 0$ for $|k| > \lambda$, then for all $R >0$ and all $N \in \NN$, there holds
    \begin{equation}
        \sup_{4 \lambda |x-x_0| \le 1} \biggl| u(x) - \sum_{|\ell| \leq R} \alpha_{\lambda,\ell}(u) e^{2\pi i \ell \cdot \lambda x} \biggr| \lesssim_{d,N} \lambda^{d/2}R^{-N}.
    \end{equation}
\end{lemma}

\begin{proof}
    Up to a translation, we may assume that $x_0 = 0$.
    Let $\chi \in C^\infty_c(\R^d)$ be supported in $\frac{1}{2} B = \{x \in \R^d : |x| \le \frac{1}{2}\}$.
    For $u \in C^\infty(\T^d)$, let
    \begin{equation*}
        g(y) = \chi (y) u \left(\frac{y}{\lambda} \right) = \chi(y) \sum_{k \in \Z^d} \widehat{u}_k e^{2 \pi i k \cdot y/\lambda}.
    \end{equation*}
    Since $\supp g \subset \frac{1}{2} B$, we may extend it periodically to $\T^d$ and still denote this extension by $g$.
    Let $(\widehat{g}_\ell)_{\ell \in \Z^d}$ be Fourier coefficients of $g$ and let $\widehat{\chi}$ be the Fourier transform of $\chi$ in $\R^d$.
    Then, define
    \begin{equation*}
        \alpha_{\lambda,\ell,x_0}(u)
        = \widehat{g}_\ell 
        = \sum_{k \in \Z^d} \widehat{u}_k \int_{\T^d} \chi(y) e^{2\pi i (k/\lambda - \ell)\cdot y} \dd y
        = \sum_{k \in \Z^d} \widehat{u}_k \widehat{\chi}\Bigl(\ell-\frac{k}{\lambda}\Bigr).
    \end{equation*}
    
    Now assume that $\widehat{u}_k = 0$ when $|k| > \lambda$.
    By the Cauchy--Schwartz inequality and the rapid decay of $\widehat{\chi}$, one deduces that, for all $N \ge 0$,
    \begin{equation*}
        |\widehat{g}_\ell|^2
        \lesssim_{\chi,N} \|u\|_{L^2}^2 \sum_{|k|\le \lambda} \Bigl(1 + \Bigl|\ell - \frac{k}{\lambda}\Bigr|^2\Bigr)^{-N}
        \le \lambda^d (1+|\ell|)^{-2N} \|u\|_{L^2}^2.
    \end{equation*}
    This implies that, for all $R \ge 1$, we have
    \begin{equation*}
        \sup_{4\lambda |x| \le 1} \biggl|u(x)-\sum_{|\ell| \le R} \widehat{g}_\ell e^{2\pi i \ell \cdot \lambda x}\biggr|
        \le \biggl\|g(y) - \sum_{|\ell| \le R} \widehat{g}_\ell e^{2\pi i \ell \cdot y}\biggr\|_{L^\infty}
        \le \sum_{|\ell| > R} |\widehat{g}_\ell|
        \lesssim_{\chi,d,N} \lambda^{d/2} R^{-N}.
        \qedhere
    \end{equation*}
\end{proof}

\begin{proof}[Proof of Proposition \ref{prop:no-observability-d-2}]
    Let $\phi$ be a Schwartz function in $\R^d$ such that $\widehat{\phi} \in C_c^\infty(\R^d)$ and satisfies $\widehat{\phi}(\xi)=1$ for $|\xi| \le 2$.
    For $\lambda \ge 1$, let $\Phi \in C^\infty(\T^d)$ be such that $\widehat{\Phi}_k = \widehat{\phi}(k/\lambda)$ for $k \in \Z^d$.
    Particularly $\widehat{\Phi}_k = 1$ when $|k| \le 2\lambda$.
    Consequently, if $u \in E_\lambda$, then
    \begin{equation*}
        \int_{\T^d} |u|^2 \dd\mu 
        = \int_{\T^d} \Phi * (|u|^2) \dd\mu
        = \int_{\T^d} |u|^2 \dd(\Phi * \mu).
    \end{equation*}
    By the Poisson summation formula and Fourier rescaling properties, one has
    \begin{equation*}
        \Phi(x) = \sum_{k \in \Z^d} \widehat{\phi}\Bigl( \frac{k}{\lambda} \Bigr) e^{2\pi i k\cdot x}
        = \lambda^d \sum_{q \in \Z^d} \phi(\lambda (x-q)).
    \end{equation*}
    From this, one immediately sees that $\Phi$ is concentrated near $\Z^d$ (i.e. $0 \in \T^d$).
    Quantitatively, letting $B_{\lambda^{-1+\epsilon} }(0) = \{x \in \T^d : \dist(x,0) \le \lambda^{-1+\epsilon}\}$, then the following estimates hold:
    \begin{itemize}
        \item $|\Phi(x)| \lesssim_{\phi,d} \lambda^d $ for $x \in B_{\lambda^{-1+\epsilon} }(0)$;
        \smallskip
        \item $|\Phi(x)| \lesssim_{\phi,d,N, \epsilon} \lambda^{-N} $ for $N \ge 0$ and $x \notin B_{\lambda^{-1+\epsilon} }(0)$.
    \end{itemize}
    Therefore, to conclude the proof it is enough to show that  there exist $\epsilon >0$ and a sequence of $L^2$-normalized eigenfunctions $u_n \in E_{\lambda_n}$ with $\lambda_n \to \infty$ such that
    \begin{equation*}
        \lim_{n \to \infty} \int_{\T^d} |u_n|^2 \dd \mu_n^\epsilon = 0,
        \quad \text{ where }
        \mu_n^\epsilon = \lambda_n^d \cdot \mathbf{1}_{B_{\lambda_n^{-1+\epsilon} }(0)} * \mu.
    \end{equation*}
     Let $\beta = \overline{\dim}_{\mathrm{M}} (\supp \mu)$ and choose $\epsilon > 0$ such that $\beta + 2d\epsilon < d-2$.
   Then, from the Minkowski dimension definition, we note that there exists $K \lesssim_{d,\mu} \lambda^{\beta + d\epsilon}$ and $(x_j)_{j=1}^K$ in $\T^d$ such that
    \begin{equation*}
        \supp \mu_n^\epsilon
        \subset B_{\lambda_n^{-1+\epsilon} }(0) + \supp \mu \subset \bigcup_{1 \le j \le K} B_{\lambda_n^{-1}/6}(x_j).
    \end{equation*}
    Let $\alpha_{\lambda,\ell,x_j}$ be defined as in Lemma~\ref{preplemma}.
    For any $n \in \NN$ we can find an $L^2$-normalized eigenfunction $u^{(n)} \in E_{\lambda_n}$ satisfying the following conditions:
    \begin{equation}
    \label{eq::alpha-vanishing-condition}
        \alpha_{\lambda_n,\ell,x_j}(u^{(n)}) = 0,
        \quad 1 \le j \le K, \ |\ell| \le \lambda^\epsilon.
    \end{equation}
    Interpreting~\eqref{eq::alpha-vanishing-condition} as a linear system in the unknowns 
\((\widehat{u}^{(n)}_k)_{k \in \mathcal{S}^{d-1}_{\lambda_n}}\), the number of unknowns is 
\(N_d(\lambda_n)\), which satisfies \(N_d(\lambda_n) \gtrsim \lambda_n^{d-2}\) for a suitably 
chosen sequence \((\lambda_n)\) (see Lemma~\ref{lem::lattice-number-lower-bound}). In contrast, 
the number of equations satisfies the bound \(\lesssim_d K \lambda_n^{d\epsilon} 
\lesssim_{d,\mu} \lambda_n^{\beta + 2d\epsilon}\).
Since $\beta + 2d\epsilon < d-2$, for this sequence and sufficiently large $n$, a nontrivial solution always exists.
    
For such $u^{(n)}$, by Lemma \ref{preplemma}, for all $j \in \{1,\ldots,K\}$ and $N \ge 0$, there holds $|u^{(n)}(x)| \lesssim_{d,N} \lambda_n^{d/2-\epsilon N}$ for $x \in B_{\lambda_n^{-1}/6}(x_j)$.
Therefore, taking $N$ large, one obtains as $n \to \infty$:
\begin{align*}
    \int_{\T^d} |u_n|^2 \dd \mu_n^\epsilon 
    & \leq \sum_{1 \le j \le K}
    \|u\|_{L^\infty(B_{\lambda_n^{-1}/6}(x_j))}^2   \times \mu_n^\epsilon(B_{\lambda_n^{-1}/6}(x_j))  \\
    & 
    \lesssim_{d,N} \sum_{1 \le j \le K} \lambda_n^{d-2\epsilon N} \times  \lambda^{ d \epsilon}
    \lesssim_{d} \lambda_n^{(\beta +  2 d\epsilon) + (d-2\epsilon N)}
    \to 0.
    \qedhere
\end{align*}
\end{proof}

\section{Pointwise Fourier decay}
\label{sectionpointwiseFourier}

\begin{theorem} \label{thm:decay-fourier}
    Let $d\geq 3$. 
    The observability inequality and the trace inequality hold true for $\mu \in \mathcal{P}(\T^d)$ if, for some $\epsilon>0$, it satisfies the Fourier decay estimate:
    \begin{equation}
    \label{eq::condition-measure-fourier-decay}
        |\widehat{\mu}_k| \lesssim |k|^{-(d-2 + \epsilon)}.
    \end{equation}
\end{theorem}

\begin{remark} \label{littleremark} 
    A typical example of measures satisfying \eqref{eq::condition-measure-fourier-decay} is given by $\dd \mu = f \dd x$ with the density $f$ being smooth except at the origin, where it behaves as $|f(x)| \sim |x|^{\epsilon-2}$, see Remark~\ref{lastremark} for a related example.
    Note that such a measure is upper $(d-2+\epsilon)$-regular.
\end{remark}

\begin{remark} \label{lastremark}
    The decay rate $d-2$ in the statement is sharp  for the trace inequality.
    More precisely, for any $\epsilon > 0$, one may construct a probability measure $\mu \in \mathcal{P}(\T^d)$ with Fourier decay $|\widehat{\mu}_k| \lesssim |k|^{-(d-2-\epsilon)}$ such that the trace inequality does not hold.
    We define the measure by setting $\dd \mu = C f \dd x$ where $C > 0$ is a normalizing constant that makes $\mu$ a probability measure and
    \begin{equation*}
        f(x) =\sum_{k \in \mathbb{Z}^d} \phi_\epsilon(x+k), \quad \forall x\in [0,1)^d \,,
        \qquad
        \phi_\epsilon(x) = e^{-|x|^2}/|x|^{2+ \epsilon} \quad \forall x \in \R^d \,.
    \end{equation*}
    A straightforward computation leads to $\widehat{f}_0 =  \widehat{\phi}_\epsilon (0) = \frac{1}{C} >0 $ and $\widehat{f}_k =\widehat{\phi}_\epsilon (k) \sim_{d,\epsilon} |k|^{-(d-2-\epsilon)} $ for all $k \in \mathbb{Z}^d \setminus \{ 0\}$.  We conclude the invalidity of the trace inequality by applying Proposition \ref{propnecessary3}, noticing that $f$ is not upper $d-2$-regular. For  $d=3,4$ one may set  $\epsilon =0$  by applying Proposition \ref{prop:d3d4}.

\end{remark}

Our proof relies on the cluster structure of  $\mathcal{S}^{d-1}_\lambda$, and more specifically on Lemma~\ref{lemma:Connes} due to Connes~\cite{Connes1976coefficients}, which enables a dimensional reduction. The core idea is as follows. If $\mathcal{S}^{d-1}_\lambda = \bigcup_\alpha \Omega_\alpha$ is the decomposition given by Lemma~\ref{lemma:Connes}, then for any $u \in E_\lambda$, we have
\begin{equation}
    \label{eq::OBS-cluster-decomp}
    \|u\|_{L^2 (d\mu)}^2
    = \sum_{k,\ell \in \mathcal{S}^{d-1}_\lambda} \widehat{\mu}_{k-\ell} \, \widehat{u}_k \, \overline{\widehat{u}_\ell}
    = \sum_\alpha \sum_{k,\ell \in \Omega_\alpha} \widehat{\mu}_{k-\ell} \, \widehat{u}_k \, \overline{\widehat{u}_\ell} 
    + \sum_{\alpha \ne \beta} \sum_{k \in \Omega_\alpha} \sum_{\ell \in \Omega_\beta} \widehat{\mu}_{k-\ell} \, \widehat{u}_k \, \overline{\widehat{u}_\ell}.
\end{equation}
Because the clusters $\Omega_\alpha$ are well separated and the measure $\mu$ exhibits Fourier decay, the cross-cluster interaction (the second term) is expected to be small and can be treated as an error term. 
It therefore suffices to establish the trace and observability inequalities for functions whose frequencies lie within a single cluster $\Omega_\alpha$. 
Since each cluster is contained in a lower-dimensional sphere, this structure naturally supports an inductive argument on the dimension.

Note that, in order for the error term to be small, one needs $\lambda$ to be sufficiently large.
More precisely, it turns out from the proof that, one only needs $\diam \supp \widehat{u}$ to be sufficiently large.
This leads to the consideration, for $r > 0$, of the following space:
\begin{equation*}
    \mathcal{T}_r
    = \{u \in C^\infty(\T^d) : \diam \supp \widehat{u} \le r\}.
\end{equation*}
In other words $\mathcal{T}_r$ is the space of all trigonometric polynomials whose Fourier supports have diameters $\le r$.
The following lemma establish quadratic estimates for trigonometric polynomials in $\mathcal{T}_r$, and it will be used in both this and the next section.

\begin{lemma}
\label{lem::trace-obs-low-freq-general}
    If { $\supp  \mu$} is not contained in the zero set of any nonzero trigonometric polynomial, then uniformaly for all $u \in \mathcal{T}_r$, there holds
\begin{equation}
\label{eq::trace-OBS-trig-diam-r}
    \|u\|_{L^2(\dd \mu)} \sim_{d,\mu,r} \|u\|_{L^2}.
\end{equation}
\end{lemma}
\begin{proof}
    Indeed, by the translation invariance of these norms in the Fourier space, we only need to consider trigonometric polynomials whose Fourier supports are within the $2r$ neighborhood of the origin.
    Such functions form a finite dimensional space. 
    Therefore, there exist the maximum and the minimum of the continuous function 
    $$u \in \{ L^2 (\T^d) : \| u \|_{L^2} =1\} \mapsto \|u\|_{L^2(\dd \mu)}^2\,.$$
    Finally, by the assumption that $\supp \mu$ is not contained in the support of any nonzero trigonometric polynomials we have that the minimum is non-zero, concluding the proof. 
\end{proof}

\begin{remark}
    This lemma applies to any $\mu \in \mathcal{P}(\T^d)$ that satisfies \eqref{eq::condition-measure-fourier-decay}.
    Indeed, by \cite[Theorem 2.8 and Section 3.6]{Mattila2015}, it implies
    \begin{equation}
        \dim_{\mathrm{H}} \supp \mu \geq  \dim_{\mathrm{F}} \supp \mu \ge 2(d-2+\epsilon) > d-1.
    \end{equation}
\end{remark}

In the following, we first present a proof for the $d=3$ case (the proof for the general case will be given afterwards).
A related result is implicit in \cite{BourgainRudnick2012}, where the authors consider measures with Fourier decay $|k|^{-1}$ and rely on the oscillatory structure of the Fourier coefficients. 
In our setting, we assume slightly faster decay, of the form $|k|^{-(1+\epsilon)}$, which enables a simpler argument.

\begin{proof}[Proof of Theorem \ref{thm:decay-fourier}  for $d = 3$.]
    By Lemma~\ref{lem::trace-obs-low-freq-general}, it remains to establish the trace and observability estimates for eigenfunctions with $\diam \supp \widehat{u} > r$ for some sufficiently large $r$.
    
    For any such $L^2$-normalized eigenfunction $u \in E_\lambda$, one has $\lambda \gtrsim r$.
    By Lemma~\ref{lemma:Connes}, write $\mathcal{S}^2_\lambda = \bigcup_\alpha \Omega_\alpha$ where each $\Omega_\alpha$ lives in an affine subspace of dimension $\le 2$ and $\dist(\Omega_\alpha,\Omega_\beta) > C \lambda^{2/4!}$ some $C>0$ whenever $\alpha \ne \beta$.
    Then we invoke the decomposition formula \eqref{eq::OBS-cluster-decomp}. We firstly estimate the second summand. Let $j_\lambda= \max\{j \in \NN : 2^j \le C \lambda^{2/4!}\}$.
    Then
    \begin{align*}
        \biggl| \sum_{\alpha\ne\beta} \sum_{k \in \Omega_\alpha} \sum_{\ell \in \Omega_\beta} \widehat{\mu}_{k-\ell} \widehat{u}_k \overline{\widehat{u}_\ell} \biggr|
        & \le \sum_{|k-\ell| > C \lambda^{2/4!}} |\widehat{\mu}_{k-\ell} \widehat{u}_k \overline{\widehat{u}_\ell}|
        \le \sum_{j \ge j_\lambda} \sum_{2^j \le |k-\ell| \le 2^{j+1}} |\widehat{\mu}_{k-\ell} \widehat{u}_k \overline{\widehat{u}_\ell}| \\
        & \lesssim \sum_{j \ge j_\lambda}  \max_{2^j \le |\xi| \le 2^{j+1}} |\widehat{\mu}_\xi| \sum_{2^j \le |k-\ell| \le 2^{j+1}} \bigl( |\widehat{u}_k|^2 + |\widehat{u}_\ell|^2 \bigr) \\
        & \lesssim \sum_{j \ge j_\lambda} N_3(\lambda,2^{j+1}) \max_{2^j \le |\xi| \le 2^{j+1}} |\widehat{\mu}_\xi|  \\
        & \lesssim_\epsilon \sum_{j \ge j_\lambda} \lambda^{\epsilon/4!} 2^j \times 2^{-j(1+\epsilon)}
        \lesssim \sum_{j \ge j_\lambda} 2^{-j\epsilon/2} = o(1)_{r \to \infty}.
    \end{align*}
    To conclude, we use the following claim: for each $\alpha$, there holds
    \begin{equation}
    \label{eq::embedded-circle-est}
        \sum_{k,\ell \in \Omega_\alpha} \widehat{\mu}_{k-\ell} \widehat{u}_k \overline{\widehat{u}_\ell}  \sim_\mu  \sum_{k \in \Omega_\alpha} |\widehat{u}_k|^2.
    \end{equation}
    Indeed, even though we are in the $d=3$ setting, the sets $(\Omega_\alpha)_\alpha$ each lies within a two dimensional affine hyperplane and is thus contained in a circle of radius $r_\alpha$.
    Therefore the exact same proof as Theorem~\ref{thm::ZBR} (for the Fourier decay case) applies.
    This proof establishes the trace inequality and the semiclassical observability, i.e., the observability inequality holds when $r_\alpha$ is sufficiently large.
    The full observability is guaranteed by Lemma~\ref{lem::trace-obs-low-freq-general}.
\end{proof}

We now turn to the general case $d \geq 3$.
Our idea is to perform mathematical induction on the affine dimension of supports of eigenfunctions.

\begin{definition}
    For any $\Omega \subset \R^d$, we denote by $\dim_{\mathrm{aff}} \Omega$ the dimension of the affine hull of $\Omega$, i.e. $\dim_{\mathrm{aff}} \Omega$ equals to the dimension of the smallest affine subspace in $\R^d$ that contains $\Omega$.
\end{definition}

\begin{proof}[Proof of Theorem \ref{thm:decay-fourier}  for $d\geq 3$.]
    By Lemma~\ref{lem::trace-obs-low-freq-general}, it remains to establish the trace and observability estimates for eigenfunctions with $\diam \supp \widehat{u} > r$ for some sufficiently large $r$.
    
    We will achieve this by using mathematical induction on the affine dimension of $\supp \widehat{u}$.
    Precisely, we will prove the following statement: if $r$ is sufficiently large, then for all $n \in \{1,\ldots,d\}$ and any toral eigenfunction $u$ satisfying  $\diam \supp \widehat{u} > r$  and $\dim_{\mathrm{aff}} \supp \widehat{u} = n$, there holds
    \begin{equation}
    \label{eq::obs-trace-induction-fourier-decay}
        \|u\|_{L^2 (\dd \mu)} \sim_{d,n,\mu,r} \| u \|_{L^2 (\T^d)}.
    \end{equation}
    We use mathematical induction.
    First we prove the case $n=1$.
    In this case  $\# \supp \widehat{u} \le 2$ and we denote $\supp \widehat{u} = \{ k, \ell \}$. 
    Let $\delta >0$ be such that $\sup_{\xi \in \Z^d \setminus 0} |\widehat{\mu} (\xi)| \leq 1 - \delta$ (Lemma \ref{lemma:delta-measure}).
    Then, as done in Section \ref{sec:dimension2}, we have
    \begin{equation*}
        \|u\|_{L^2}^2 \gtrsim \int_{\T^d} |u|^2 \dd \mu = |\widehat{u}_k|^2 + |\widehat{u}_\ell|^2 + 2 \widehat{\mu}_{k-\ell} \widehat{u}_k \widehat{u}_\ell \geq \delta ( |\widehat{u}_k|^2 + |\widehat{u}_\ell|^2 )
    = \delta \|u\|_{L^2}^2.
    \end{equation*}

    We now prove the induction step $n \Rightarrow n+1$ with $n \leq d-1$.
    Fix any $L^2$-normalized eigenfunction $u \in E_\lambda$ with $\dim_{\mathrm{aff}} \supp \widehat{u} = n+1$.
    Now $\supp \widehat{u}$ is contained in an $n$-sphere of radius $\rho \gtrsim r$.
    By Lemma \ref{lemma:Connes}, we write $\supp \widehat{u} = \bigcup_\alpha \Omega_\alpha$ where $\dim_{\mathrm{aff}} \Omega_\alpha \le n$ for all $\alpha$ and $\dist(\Omega_\alpha,\Omega_\beta) \ge C_n \rho^{2/(n+1)!}$ for some $C_n > 0$ whenever $\alpha \ne \beta$.
    Then we invoke the decomposition formula \eqref{eq::OBS-cluster-decomp} and for each $\alpha$, we use the induction hypothesis and obtain
    \begin{equation*}
        \sum_\alpha \sum_{k,\ell \in \Omega_\alpha} \widehat{\mu}_{k-\ell} \widehat{u}_k \overline{\widehat{u}_\ell}
        \sim_{d,n,\delta} \sum_\alpha  \sum_{k \in \Omega_\alpha} |\widehat{u}_k|^2
        = \|u\|_{L^2}^2.
    \end{equation*}
    To estimate the second summand term in \eqref{eq::OBS-cluster-decomp}, we let $j_s = \max\{j \in \NN : 2^j \le C_n s^{2/d!}\}$ and write
    \begin{align*}
        \sum_{|k-\ell| > c_{d-1} \lambda^{2/d!}} |\widehat{\mu}_{k-\ell} \widehat{u}_k \overline{\widehat{u}_\ell}|
        & \le \sum_{j \ge j_\rho} \sum_{2^j \le |k-\ell| \le 2^{j+1}} |\widehat{\mu}_{k-\ell} \widehat{u}_k \overline{\widehat{u}_\ell}| \\
        & \lesssim \sum_{j \ge j_\rho} \max_{2^j \le |\xi| \le 2^{j+1}} |\widehat{\mu}_\xi| \sum_{2^j \le |k-\ell| \le 2^{j+1}}  \bigl( |\widehat{u}_k|^2 + |\widehat{u}_\ell|^2 \bigr) \\
        & \lesssim \sum_{j \ge j_\rho} \sup_{k \in \supp \widehat{u}} \# \bigl(\supp \widehat{u} \cap B_{2^{j+1}}(k)\bigr) \times \max_{2^j \le |\xi| \le 2^{j+1}} |\widehat{\mu}_\xi|.
    \end{align*}
    Next, we separate the discussion of the two cases: $n=d-1$ and $n<d-1$:
    \begin{itemize}
        \item If $n=d-1$, then $\# \bigl(\supp \widehat{u} \cap B_{2^{j+1}}(k)\bigr) \le N_d(\lambda,2^{j+1}) \lesssim_d \lambda^\epsilon 2^{j(d-2)}$ and $\lambda =\rho \ge r$.
        Therefore, the error term is bounded by
        \begin{equation*}
            \sum_{j \ge j_\lambda} \lambda^{ \epsilon/d!} 2^{j(d-2)} \times 2^{-j(d-2+\epsilon)}
            \lesssim \sum_{j \ge j_\lambda} 2^{-j\epsilon/2} 
            = \sum_{j \ge j_r} 2^{-j\epsilon/2} 
            = o(1)_{r \to \infty}.
        \end{equation*}
    
        \item If $n \leq d-2$, then $\# \bigl(\supp \widehat{u} \cap B_{2^{j+1}}(k)\bigr) 
            \lesssim_n 2^{n j} \lesssim 2^{(d-2) j}$ by the area method.
        Therefore, the error term is bounded by
        \begin{equation*}
            \sum_{j \ge j_\rho} 2^{(d-2)j} \times 2^{-j(d-2+\epsilon)} 
            \lesssim \sum_{j \ge j_\rho} 2^{-j\epsilon} 
            = \sum_{j \ge j_r} 2^{-j\epsilon}
            = o(1)_{r \to \infty}.
        \end{equation*}
    \end{itemize}
    
    We finish the induction process by combining these estimates above and formula \eqref{eq::OBS-cluster-decomp}.
    They yield the following estimate:
    \begin{equation*}
        \sum_{k,\ell \in \mathcal{S}^{d-1}_\lambda} \widehat{\mu}_{k-\ell} \widehat{u}_k \overline{\widehat{u}_\ell}
        \sim_{d,n,\delta} 1 + o(1)_{r \to \infty}.
        \qedhere
    \end{equation*}
\end{proof}

When $ d \ge 5$, we can relax the decay condition on $\mu$.

\begin{theorem}
    For $d \ge 5$, the trace and observability inequalities hold if $\mu$ satisfies
    \begin{equation}
    \label{eq::mu-condition-dyadic-proof}
        \sum_{j \ge 0} 2^{j(d-2)} \sup_{|\xi| \in [2^j,2^{j+1}]} |\widehat{\mu}_\xi| < \infty.
    \end{equation}
\end{theorem}
\begin{proof}
    The proof follows the same outline as before, with only minor adjustments noted below:
    \begin{itemize}
        \item The condition \eqref{eq::mu-condition-dyadic-proof} implies the Fourier decay $|\widehat{\mu}_\xi| = o(|\xi|^{-(d-2)})_{\xi \to \infty}$.
        Therefore we still have $\dim_{\mathrm{H}} \supp \mu > d-1$ to apply Lemma~\ref{lem::trace-obs-low-freq-general}.
        
        \item In the $n=d-1$ case, we use the finer estimate $N_d(\lambda,2^{j+1}) \lesssim_{d,\epsilon} \lambda^{-1} 2^{j (d-1)} + \lambda^{\epsilon/d!} 2^{j(d-3+\epsilon)}$.
        Since it suffices to sum over those $j$ with $ { \lambda^{2/d!} } \lesssim 2^j \lesssim \lambda$, we may further bound $N_d(\lambda,2^{j+1}) \lesssim 2^{j(d-2)}$.
        Therefore, the error term is bounded by
        \begin{equation*}
            \sum_{j \ge j_\lambda} 2^{j(d-2)} \sup_{2^j \le |\xi| \le 2^{j+1}} |\widehat{\mu}_\xi|
            \le \sum_{j \ge j_r} 2^{j(d-2)} \max_{2^j \le |\xi| \le 2^{j+1}} |\widehat{\mu}_\xi| = o(1)_{r \to \infty}.
        \end{equation*}
        
        \item In the $n\leq  d-2$ case, it suffices to notice that, by keeping $\max_{2^j \le |\xi| \le 2^{j+1}} |\widehat{\mu}_\xi|$ in the summation, the remainder term satisfies the same estimate:
        \begin{equation*}
            \sum_{j \ge j_r} 2^{j(d-2)} \sup_{2^j \le |\xi| \le 2^{j+1}} |\widehat{\mu}_\xi| = o(1)_{r \to \infty}.\qedhere
        \end{equation*}
    \end{itemize}
\end{proof}

\section{Sobolev regularity}
\label{sectionsobolev}

In this section we prove observability and trace inequalities under suitable Sobolev regularity assumptions of the measure.
They are closely related to the estimates with subpolynomial loss 
\begin{equation*}
    \| u \|_{L^p} \lesssim \lambda^\epsilon \| u \|_{L^2},
    \quad p \le \frac{2(d+1)}{d-1},
\end{equation*}
for any eigenfunction $u \in E_\lambda$. If $d=3$, this estimate was mentioned without proof in \cite{Bourgain1997}, and was generalized to $d \geq 3$ through the $\ell^2$ decoupling theorem \cite{BourgainDemeter}.

\begin{theorem} \label{thm:sobolev}
    For $d \geq 3$, the trace and observability inequality holds true if $ \dd \mu = g \dd x$ where $g \in W^{\epsilon,\frac{d+1}{2}} (\T^d)$ for some $\epsilon>0$.
    Furthermore, under the same assumption, this trace inequality is uniform in the sense that for all toral eigenfunctions:
    \begin{equation}
    \label{eq::trace-uniform-sobolev}
        \int_{\T^d} |u|^2 \dd \mu \lesssim_{ d,\epsilon} \| g \|_{W^{\epsilon , \frac{d+1}{2}  }} \| u \|_{L^2}^2.
    \end{equation}
\end{theorem}

\begin{remark}
    It is interesting and relevant in applications to apply this theorem to characteristic functions of sets, i.e., $\dd \mu = \mathbf{1}_E \dd x$.
    In Appendix \ref{sec:appendixA}, we explore the case where $C$ is a fat Cantor set and prove that $\mathbf{1}_C \in W^{\epsilon, 1}$ for suitable $\epsilon >0$. See also  \cite{Lomb19} for the fractional Sobolev regularity of $\mathbf{1}_S$ where  $S \subset \R^2$ is the von Koch snowflake.
\end{remark}

\begin{remark}
    Since now the measure $\mu$ is absolutely continuous with respect to the Lebesgue measure on $\T^d$, Lemma~\ref{lem::trace-obs-low-freq-general} applies.
\end{remark}

We chose to provide first a separate proof for the case $d = 3$ which relies on elementary tools. It is then generalized to the case $d \geq 3$ with the help of the $\ell^2$ decoupling theorem \cite{BourgainDemeter}.

\begin{proof}[Proof of Theorem \ref{thm:sobolev} for $d=3$]
    The proof follows the same structure as that of Theorem~\ref{thm:decay-fourier}, specifically the proof for the $d=3$ case.
    That is, we apply the cluster decomposition and invoke the formula \eqref{eq::OBS-cluster-decomp}.
    We now  estimate the second summand in \eqref{eq::OBS-cluster-decomp}. We first observe that, by Lemma~\ref{lemma:num-theo-sphere}, for all $\epsilon > 0$ it holds
    \begin{align*}
        \sum_{|k- \ell|> C \lambda^{2/4!}}|\widehat{\mu}_{k-\ell}|^2  & \leq  \sum_{|\xi| > C \lambda^{2/4!}} \# \Bigl(
    (-\xi + \mathcal{S}^2_\lambda) \cap \mathcal{S}^2_\lambda \Bigr) \times |\widehat{\mu}_{\xi}|^2  \\
    & \lesssim_\epsilon \lambda^{\epsilon/4!} \sum_{|\xi| > C \lambda^{2/4!}}   |\widehat{\mu}_{\xi}|^2
    \lesssim \sum_{|\xi| > C \lambda^{2/4!}}  |\xi|^{2\epsilon} |\widehat{\mu}_{\xi}|^2.
    \end{align*}
    Hence, by the Cauchy Schwarz inequality and using $g \in H^\epsilon$,
    \begin{align*}
    \sum_{|k- \ell| > C \lambda^{2/4!}} |\widehat{\mu}_{k-\ell} \widehat{u}_k \overline{\widehat{u}_\ell}|  
    & \leq \biggl ( \sum_{|k- \ell| > C \lambda^{2/4!}} |\widehat{\mu}_{k-\ell}|^2 \biggr )^{1/2}  \biggl (\sum_{|k- \ell| > C \lambda^{2/4!}} | \widehat{u}_k \overline{\widehat{u}_\ell}|^2 \biggr )^{1/2} \\
    & \lesssim_\epsilon \biggl ( \sum_{|\xi| > C \lambda^{2/4!}} |\xi|^{2 \epsilon} | \widehat{\mu}_{\xi}|^2 \biggr )^{1/2} \| u \|_{L^2}^2
    = o(1)_{\lambda \to \infty}.
    \end{align*}
    Then for each $\alpha$, we obtain the estimate \eqref{eq::embedded-circle-est} by the exact same proof as in Theorem~\ref{thm::ZBR} (the case with $L^2$ densities).
    Particularly, the constant for the trace inequality is bounded by $\|g\|_{H^\epsilon}$.
\end{proof}

We now turn to the general case $d \geq 3$ using the decoupling theory \cite{BourgainDemeter}.

\begin{proof}[Proof of Theorem \ref{thm:sobolev} for $d \geq 3$]
    It remains to establish \eqref{eq::obs-trace-induction-fourier-decay} for any toral eigenfunction $u$ satisfying $\diam \supp \widehat{u} > r$, where $r$ is sufficiently large, and $\dim_{\mathrm{aff}} \supp \widehat{u} = n$.
    Regarding the uniform trace inequality, we will also show that \eqref{eq::trace-uniform-sobolev} for such eigenfunctions.
    
    We use mathematical induction. The case $n=1$ is the same as in the proof of Theorem \ref{thm:decay-fourier}.
    We now prove the inductive step $n \Rightarrow n+1$ with $n \leq d-1$. We perform the cluster decomposition on the $n$-sphere that supports $\widehat{u}$, where $u$ is an $L^2$-normalized eigenfunction.
    Let $\rho$ be the radius of this sphere, then $\rho \gtrsim r$.
    Recall that these clusters are well-separated: if $\alpha \ne \beta$, then $\dist(\Omega_\alpha,\Omega_\beta) > N = C_n \rho^{2/(n+2)!} $.
    Let $\mathbb{P}_N$ be a smooth spectral projector onto the frequency set $\{k \in \Z^d : |k| \le N\}$, and write
    \begin{equation}
    \label{splitting}
    \int_{\T^d} |u|^2 g \dd x 
    = \int_{\T^d} |u|^2 \mathbb{P}_N g \dd x 
    + \int_{\T^d} |u|^2 (1-\mathbb{P}_N) g \dd x.
    \end{equation}
    
    To bound the first term on the right-hand side of \eqref{splitting}, we use
    \begin{align*}
    \int_{\T^d} |u|^2 { \mathbb{P}_N g} \dd x 
    = \sum_{\alpha} \sum_{k, \ell \in \Omega_{\alpha}} \widehat{\mathbb{P}_N g}_{k-\ell} \widehat{u}_k \overline{\widehat{u}_\ell} .
    \end{align*}
    Since $\Omega_\alpha$ is contained in an affine $n$-plane, we may apply the inductive hypothesis for each $\alpha$
    \begin{align} \label{eq:lower-induct}
         \sum_{k \in \Omega_\alpha} |\widehat{u}_k|^2
         \lesssim_{d,n,\delta}
         \sum_{k, \ell \in \Omega_{\alpha}} \widehat{\mathbb{P}_N g}_{k-\ell} \widehat{u}_k \overline{\widehat{u}_\ell}  \lesssim_{d,n, \delta} \|\mathbb{P}_N g\|_{W^{\epsilon,\frac{d+1}{2}}} \sum_{k \in \Omega_\alpha} |\widehat{u}_k|^2.
    \end{align}
    We now bound the second term in the right-hand side of \eqref{splitting}. We first observe that 
    for $s \ge 0$ and $p\in [1,\infty)$, there holds
    \begin{equation*}
     \|\mathbb{P}_N\|_{W^{s,p} \to W^{s,p}} \lesssim_{d,s,p} 1,
     \quad
     \| (1- \mathbb{P}_N)g \|_{W^{s,p} \to L^p} \lesssim_{d,s,p} N^{-s}.
    \end{equation*}
    Using this, the H\"older's inequality, the decoupling theorem 
    \cite[Theorem 2.2]{BourgainDemeter} we conclude 
    \begin{align*}
    \int_{\T^d} |u |^2 (1-\mathbb{P}_N) (g) \dd x 
    & \le \| u \|_{L^{ \frac{d+1}{d-1}}}^2 \| (1-\mathbb{P}_N) g \|_{ L^{\frac{d+1}{2}}} 
    \lesssim_{d,\epsilon} N^{\epsilon/2}  \| u \|_{L^2}^2 \| (1-\mathbb{P}_N)g \|_{ L^{\frac{d+1}{2}}}
    \\
    & \lesssim_{d,\epsilon} N^{- \epsilon/2}  \| u \|_{L^2}^2 \| (1-\mathbb{P}_N)g \|_{ W^{\epsilon, \frac{d+1}{2}}} \lesssim \| u \|_{L^2}^2 \| g \|_{ W^{\epsilon, \frac{d+1}{2}}} = o(1)_{r \to \infty} \,.
    \qedhere
    \end{align*}
\end{proof}

\appendix

\section{Sobolev regularity of sets} \label{sec:appendixA}

\subsection{The basic question}

From our discussions in Section~\ref{sectionperspectives}, a natural question is to characterize the Sobolev regularity of measurable sets. 
Recall that, for $s \in [0,\infty)$, $p\in [0,\infty)$ and an open set $G \subset \R^d$, the fractional Sobolev space $W^{s,p}(G)$ is defined as
\begin{equation*}
    W^{s,p}(G) = \bigl\{ f \in L^p(G) : [f]_{W^{s,p}} < \infty \bigr\},
\end{equation*}
where $[f]_{W^{s,p}}$ is the Gagliardo seminorm given by
\begin{equation}
\label{eq::Gagliardo-seminorm}
    [f]_{W^{s,p}} \coloneqq \left( \int_G \int_G \frac{|f(x) - f(y)|^p}{|x-y|^{d + sp}} \dd x \dd y \right)^{1/p}.
\end{equation}
Applying \eqref{eq::Gagliardo-seminorm} to the characteristic function of $E \subset G$, we see that if $sp< 1$, then
\begin{equation*}
    \| \mathbf{1}_E \|_{[W^{s,p}]} = \left[ 2\int_{E} \int_{G\setminus E} \frac{1}{|x-y|^{sp+d} } \dd x \dd y \right]^{1/p} = \| \mathbf{1}_E \|_{[W^{sp,1}]}.
\end{equation*}
Since we are interested in the range $sp< 1$, we will focus in the following on the case $p=1$.

\subsection{The example of the fat Cantor set}

We construct the fat Cantor set, a modified version of the classical ternary Cantor set such that at the $(n+1)$-th step we remove $2^n$ sets of size $\alpha^{n+1}$ where  $\alpha \in (0, 1/3)$. In the classical construction of the Cantor set one chooses $\alpha=1/3$. 

We construct the fat Cantor set with parameter $\alpha$ by an iterative process. 
Let $C_0 = [0,1] $. 
Suppose that for $n \in \NN$, we are given $C_n$ as a disjoint union of $2^n$ closed intervals $ \{ C_{k, n} \}_{k=1}^{2^n}$ of length larger than $\alpha^n$.
Let $I_{k, n} $ is the unique open subinterval of $C_{k, n}$ with the same center of $C_{k, n}$ and of length $\alpha^{n+1}$.
Then write
\begin{equation*}
    C_{k, n} \setminus I_{k, n} = C_{2k-1,n+1} \cup C_{2k,n+1}
\end{equation*}
as a disjoint union of closed intervals with length larger than $\frac{1}{2}(\alpha^n - \alpha^{n+1}) > \alpha^{n+1}$.
Then define $C_{n+1} \subset C_n$ as the union of $ \{C_{k,n+1}\}_{k=1}^{2^{n+1}}$.
Finally, we define 
\begin{align} \label{eq:cantor-int}
     \mathcal{C}_\alpha = \bigcap_{n \ge 0} C_n
     = \bigcap_{n \ge 0} \bigcup_{1 \le k \le 2^n} C_{k, n}
     = C_0 \setminus \bigcup_{n \ge 0} \bigcup_{1 \le k \le 2^n} I_{k, n} .
\end{align}

\begin{lemma}
    If $\alpha \in (0,1/3)$ and $\epsilon \in  (0, 1 + \ln 2 / \ln \alpha )$, then $\mathbf{1}_{\mathcal{C}_\alpha \setminus \{0,1\}} \in W^{\epsilon, 1} (0,1)$.
\end{lemma}
\begin{proof}
    Let $\mathcal{C} = \mathcal{C}_\alpha \setminus \{0,1\}$.
    By \eqref{eq::Gagliardo-seminorm} and \eqref{eq:cantor-int}, we have
    \begin{equation*}
        \| \mathbf{1}_{\mathcal{C}} \|_{[W^{\epsilon, 1}]} 
    = \int_{\mathcal{C}} \int_{(0,1)\setminus \mathcal{C}} \frac{1}{|x-y|^{1+\epsilon}} \dd x \dd y   = \| \mathbf{1}_{(0,1)\setminus \mathcal{C}} \|_{[W^{\epsilon, 1}]} 
    \leq \sum_{n\ge 0} \sum_{1 \le k \le 2^n}  \| \mathbf{1}_{I_{k, n}} \|_{[W^{\epsilon, 1}]}.
    \end{equation*}
    For each term in the summation, we compute 
    \begin{align*}
        \| \mathbf{1}_{I_{k, n}} \|_{[W^{\epsilon, 1}]}  
        & \leq \int_{I_{k, n}} \int_{[0,1] \setminus I_{k, n}} \frac{1}{|x-y|^{1+\epsilon}} \dd x \dd y 
        \leq \int_0^{\alpha^{n+1}} \int_{\alpha^{n+1}}^1 \frac{1}{|x-y|^{1+\epsilon}} \dd x \dd y
        \\
        & =  \frac{\alpha^{(n+1)(1- \epsilon)} + (1- \alpha^{n+1})^{1- \epsilon} -1 }{\epsilon (1- \epsilon)}
        \leq \frac{\alpha^{(n+1)(1- \epsilon)}}{\epsilon (1- \epsilon)}  .
    \end{align*}
    Therefore, if $\epsilon \in  (0, 1 + \ln 2 / \ln \alpha )$, then
    \begin{equation*}
        \| \mathbf{1}_{\mathcal{C}} \|_{[W^{\epsilon, 1}]} \leq \frac{1}{\epsilon (1- \epsilon)} \sum_{n \ge 0} 2^{n+1} \alpha^{(n+1)(1- \epsilon)} < \infty.
        \qedhere
    \end{equation*}
\end{proof}

\subsection{An example of an irregular  set}

For $n \ge 3$, let $h_n = \frac{1}{n (\ln n)^2}$, then $\sum_{n \ge 3} h_n \le \frac{1}{\ln 2}$.
We may construct a sequence $(a_n)_{n \geq 3}$ such that $a_{n-1} - a_n= 2 h_n$ and $\lim_{n \to \infty} a_n = 0$.
Then, let
\begin{equation*}
    A= \bigcup_{n \geq 3} ( a_n, a_n + h_n ) \subset (0,3).
\end{equation*}

\begin{lemma}
    We have $\mathbf{1}_A \notin W^{\epsilon, 1} (0,3)$ for any $\epsilon >0$.
\end{lemma}
\begin{proof}
    Since $a_{n-1} - a_n= 2 h_n$, the following estimate holds true 
    \begin{align*}
        \| \mathbf{1}_A \|_{[W^{\epsilon, 1}]} & \geq \sum_{n \geq 4} \int_{a_{n}}^{a_n + h_n} \int_{a_n + h_n}^{a_{n-1}}  \frac{1}{|x-y|^{1+ \epsilon}} \dd x \dd y
        \geq \sum_{n \geq 4} \int_{a_{n}}^{a_n + h_n} \int_{a_n+h_n}^{a_n+2h_n} \frac{1}{|x-y|^{1+ \epsilon}} \dd x \dd y
        \\
        & = \sum_{n \geq 4} \int_{0}^{ h_n} \int_{ h_n}^{2h_n} \frac{1}{|x-y|^{1+ \epsilon}} \dd x \dd y
        = \frac{2 - 2^{1- \epsilon}}{\epsilon (1- \epsilon)} \sum_{n \geq 4}  h_n^{1- \epsilon}  = \infty .
        \qedhere
    \end{align*}
\end{proof}

\printbibliography

\end{document}